 \documentclass[11pt,twoside,reqno,centertags]{amsart}

 \usepackage{geometry}
\usepackage{graphicx}
\usepackage{amsmath,amssymb}
\usepackage{enumerate}
\usepackage{hyperref}
\usepackage[active]{srcltx}

\usepackage[T1]{fontenc}
\usepackage{color}

\newtheorem{theo}{Theorem}
\newtheorem{lem}{Lemma}[section]
\newtheorem{defi}[lem]{Definition}
\newtheorem{cor}[lem]{Corollary}
\newtheorem{prop}[lem]{Proposition}
\newtheorem{rmk}[lem]{Remark}

\newcommand{\eps}{\varepsilon}
\newcommand{\R}{\mathbb{R}}
\newcommand{\N}{\mathbb{N}}
\newcommand{\Mm}{\mathcal{M}^+}
\renewcommand{\o}{\overline}
\newcommand{\p}{\partial}
\numberwithin{equation}{section}
\DeclareMathOperator{\dive}{div}
\DeclareMathOperator{\grad}{grad}

\newcommand{\rd}{\mathrm{d}}

\renewcommand{\u}{{\mathfrak u}}


 \usepackage{amsmath,amsthm,amsfonts,amssymb}
\usepackage[mathscr]{euscript}
 \usepackage{hyperref} 
 \usepackage{mathrsfs} 
\usepackage{graphicx}
\usepackage{color}
  \usepackage{epsfig}

\begin{document}
 
\title{ A new optimal transport distance \\
 on the space of finite Radon measures }

\date{}
\maketitle     
 
\vspace{ -1\baselineskip}

{\small
\begin{center}
{\sc Stanislav Kondratyev}
\\ 
CMUC, Department of Mathematics, University of Coimbra\\ 
3001-501 Coimbra, Portugal 
\\[10pt]
{\sc L\'eonard Monsaingeon}
\\
CAMGSD, Instituto Superior T\'ecnico, University of Lisbon\\
1049-001 Lisboa, Portugal
\\[10pt]
{\sc Dmitry Vorotnikov}
\\ 
CMUC, Department of Mathematics, University of Coimbra\\
3001-501 Coimbra, Portugal \\[10pt]

\end{center}
}

\numberwithin{equation}{section}
\allowdisplaybreaks

 \smallskip

 \begin{quote}
\footnotesize
{\bf Abstract.}  
We introduce a new optimal transport distance between nonnegative
finite Radon measures with possibly different masses. The construction is based
on non-conservative continuity equations and a corresponding modified 
Benamou-Brenier formula. We establish various topological and geometrical properties of
the resulting metric space, derive some formal Riemannian structure, and develop
differential calculus following F. Otto's approach. Finally, we apply these ideas to
identify a model of animal dispersal proposed by MacCall and Cosner as a gradient
flow in our formalism and obtain new long-time convergence results.

\end{quote}

\section{Introduction} 
In the last decades the theory of Monge-Kantorovich optimal transportation problems has seen spectacular developments and provided new powerful tools, deep insights, and numerous applications in functional analysis, partial differential equations, and geometry. The list of references on this topic is steadily growing and we refer to \cite{villani08oldnew} for an extended bibliography. Central ingredients to the theory are the Kantorovich-Rubinstein-Wasserstein distances and their variants, usually defined between probability measures. In their seminal paper \cite{JKO}, Jordan, Kinderlehrer and Otto showed that certain diffusion equations can be interpreted as gradient flows with respect to the Wasserstein metric structure. This was later pushed further by Otto \cite{otto01}, who showed that the space of probability measures can in fact be endowed with a formal Riemannian structure induced by the (quadratic) Wasserstein distance. This differential structure then led to new connections between long time 
convergence of measure dynamics and functional inequalities via, e.g., entropy dissipation methods, Talagrand and logarithmic Sobolev inequalities, Bakry-\'Emery strategies. See \cite{CMcV03,villani03topics} for a review, and also Section~\ref{section:Riemannian_structure} for a brief discussion.

One of the main restrictions of the Wasserstein distances is that they are limited to measures with fixed identical masses, and the theory requires uniform tightness (control of decay at infinity via the $p$-moments).
Recently, some efforts \cite{FG,PR13properties,PR14generalized} were made to construct new optimal transport distances between measures with different masses, and some non-mass preserving reaction-diffusion systems were interpreted as gradient flows \cite{CL15,LM13}.

In this paper we introduce a new distance on the set of nonnegative finite Radon measures by means of a modified Benamou-Brenier formula. 
Our approach allows for mass variations and does not require tightness or decay conditions such as finite moments.
The classical Benamou-Brenier formula \cite{AGS06,BenamouBrenier00,villani03topics} has an interpretation that the squared quadratic Wasserstein distance $\mathcal{W}^2_2(\rho_0,\rho_1)$ between two probability measures $\rho_0$ and $\rho_1$ (with finite second moments) is the minimum of the Lagrangian action of the kinetic energy during all possible ways of transporting the original distribution $\rho_0$ of moving particles to the target one $\rho_1$ via \emph{continuity equations} $\p_t\rho+\dive(\rho \mathbf v)=0$.  In a similar spirit, our distance has two interpretations: a mechanical one through motion of charged particles (described later on in Section \ref{tdg}), and a biological one through fitness-driven dispersal of organisms described below. For both points of view the crucial role is played by the \emph{non-conservative continuity equation}
\begin{equation*}
\label{eq:continuity-in}
\partial_t\rho+\dive(\rho\nabla u)=\rho u,
\end{equation*}
which allows for mass variations through the reaction term $\rho u$ appearing in the right-hand side.
Biologically, $\rho(t,x)$ can be viewed as the time evolution of the spatial density of living organisms, and $u(t,x)$ as an intrinsic characteristic of population called \emph{fitness} (cf. \cite{cosner05,cos13}). 
The fitness manifests itself as a growth rate, and simultaneously affects the dispersal, as the species move along the gradient towards the most favorable environment. The equilibrium $u\equiv 0$ is called the \emph{ideal free distribution} \cite{FC69,fr72}, since no net movement of individuals occurs in this case. Our (squared) distance is the minimum of the Lagrangian action of the total energy, which is the sum of the kinetic energy $\rho|\nabla u|^2$ and of the potential energy $\rho|u|^2$ representing deviation from the ideal free distribution.

The advantages of our distance are that it possesses a rich underlying geometric structure and is easy to handle heuristically by the optimal transport intuition. It metrizes the narrow convergence of measures, and is lower-semicontinuous with respect to the weak-$*$ convergence.  We prove that the metric space under consideration is a complete geodesic space. The compactly supported measures are dense in this space. The Lipschitz curves (in particular, the geodesics) in this space have a clear characterization. Our distance endows the space of finite Radon measures with a formal Riemannian structure, and we are able to introduce a first- and second-order calculus \emph{\`{a} la Otto}. 

The fitness-driven dispersal model suggested in \cite{mc90,cosner05} and studied in \cite{CW13} (see also \cite{ccl08}) turns out to be a gradient flow in our formalism. This allows us to show that the solutions to this problem exponentially converge to the ideal free distribution. The convergence itself was proven in \cite{CW13} by contradiction and thus without any rate. Related fitness-driven two-animal models were investigated in \cite{ccl13,ltw14}.  In forthcoming papers \cite{KMV16-1,KMV16-2} we study an ecological model for several interacting animal populations by observing that it is also a gradient flow with respect to our structure. 
We also refer the reader to \cite{carrillo2012structured,evers2015mild,gwiazda2010nonlinear,gwiazda2010structured,hille2009embedding,ulikowska2012age} and references therein for applications of metric structures in spaces of measures (and, in particular, of the bounded-Lipschitz distance), e.g. in the context of population dynamics, pedestrian flows, and Markov chains.

For simplicity we restricted here to the quadratic cost $\rho (|\nabla u|^2+|u|^2)$, leading to the aforementioned formal Riemannian structure in the spirit of Otto.
One could otherwise choose costs of the form $\rho |\nabla u|^p,\rho|u|^q$ for different exponents $p,q>1$ and construct a whole family of distances $d_{p,q}$, but those would not enjoy the same fashionable Riemannian structure (one should then rather speak of tangent \emph{cones}).
This is similar to the dynamic formulation of the Wasserstein distances $\mathcal W_p$ of order $p$ \cite{AGS06,villani03topics}, and the gradient flows with suitable driving entropies can involve the $p$-Laplacian operator as in \cite{agueh2005existence}.
This might be applied to identify different nonlinear reaction-diffusion and population models as gradient flows with respect to some $d_{p,q}$ distances, yet those metrics can be useful in various other applications involving quantities of variable mass.

In \cite{dolbeault_nazaret_new_2009} the authors constructed a new class of pseudo-distances on the space of non-negative Radon measures via modified Benamou-Brenier formulas, and as a result some techniques are similar to the ones here. They employed \emph{conservative continuity equations} but with nonlinear mobilities. On the other hand, a  \emph{non-conservative} modifed Benamou-Brenier formula (different from ours) appears in \cite{PR13properties}. More specifically, a class of distances $W_p$ on the space of non-negative Radon measures  is constructed in  \cite{PR13properties,PR14generalized} using perturbations of the original measures in order to obtain measures of equal mass, and consequent minimization of  the sum of the classical Kantorovich-Wasserstein optimal transport cost and of the price of the perturbations which is measured with the help of the total variation distance.  These distances, exactly as ours, metrize the narrow convergence of measures. The distance $W_1$ coincides with the 
bounded-Lipschitz distance, and the distance $W_2$ between absolutely continuous measures can be calculated by minimizing a Lagrangian in the manner of Benamou-Brenier with the energy $|h|+\rho |\mathbf{v}|^2$, where $\partial_t\rho+\dive(\rho\mathbf{v})=h$.

The paper is organized as follows: The new distance is defined in Section~\ref{section:metric_space}, where we establish the aforementioned topological properties and characterize Lipschitz curves and geodesics. Section \ref{section:Riemannian_structure} is devoted to the Riemannian formalism and differential calculus {\it \`a la Otto}. We introduce the notion of \emph{trajectory geodesics} allowing us to develop the second-order calculus via some Hamilton-Jacobi equations, and also present some explicit and illustrative computations for one-point measures. In Section~\ref{section:animals} we exploit this Riemannian formalism to identify the fitness-driven ideal-free distribution model \cite{mc90,cosner05,CW13} as a gradient flow and retrieve new long-time convergence results. Finally, we opted for moving several auxiliary results, including a new entropy-entropy production inequality, to the Appendix (Section~\ref{section:appendix}).
%


\subsection*{Note during final preparation.}
After finalization of this article we became aware of the parallel and completely independent works \cite{peyre_1_2015,LMS_big_2015}, whose preprints appeared almost simultaneously with ours.
In these studies, the same distance is constructed, but with different approaches and techniques, and different types of results are proved.
We also refer to \cite{LMS_small,peyre_2} for some extensions and applications.

In \cite{peyre_1_2015}, the authors considered the continuity equations $\partial_t\rho+\dive(\rho\mathbf{v})=\rho u$ for \emph{independent} velocity fields  and reaction rates, and the dynamical cost is then minimized among all couples $(u,\mathbf{v})$.
It is not difficult to check at least formally that any minimizer $(u,\mathbf{v})$ must satisfy $\mathbf{v}=\nabla u$.
Thus the distance constructed therein is the same as ours, even though we restrict ourselves to potentials $(u,\nabla u)$ from the beginning.
\pagebreak
\subsection*{Notation and conventions}
\begin{itemize}
 \item 
Throughout the whole paper and unless otherwise specified we will always denote by
$$
\Mm=\mathcal{M}^+_b(\R^d)
$$
the set of nonnegative finite Radon measures in $\R^d$.
\item
We use the following notation for sets of functions:
\begin{center}
\begin{tabular}{ll}
$\mathcal{C}_b$: & bounded continuous with $\|\phi\|_{\infty}=\sup |\phi|$; \\
$\mathcal{C}_b^1$: & bounded $\mathcal{C}^1$ with bounded first derivatives;\\
$\mathcal{C}^\infty_c$: & smooth compactly supported;\\
$\mathcal{C}_0$: & continuous and decaying at infinity;\\
$\operatorname{Lip}:$ & bounded and Lipschitz with $\|\phi\|_{\operatorname{Lip}}=\|\nabla\phi\|_{\infty}+\|\phi\|_{\infty}$.
\end{tabular}
\end{center}
\item
Given a sequence $\{\rho^k\}_{k\in\N}\subset \Mm$ and $\rho\in\Mm$ we say that:
\begin{enumerate}[(i)]
 \item $\rho^k$ \emph{converges narrowly} to $\rho$ if there holds
 $$
 \forall\,\phi\in \mathcal{C}_b(\R^d):\qquad \lim\limits_{k\to\infty}\int_{\R^d}\phi(x) \rd \rho^k(x)=\int_{\R^d}\phi(x) \rd \rho(x).
 $$
\item $\rho^k$ \emph{converges weakly-$*$}  to $\rho$ if there holds
 $$
 \forall\,\phi\in \mathcal{C}_0(\R^d):\qquad \lim\limits_{k\to\infty}\int_{\R^d}\phi(x) \rd \rho^k(x)=\int_{\R^d}\phi(x) \rd \rho(x).
 $$
\end{enumerate}

\item
Given a measure $\rho_0\in \Mm$ and a continuous function $F:\R^d\to\R^d$, the measure $F\#\rho_0$ is the pushforward of $\rho_0$ by $F$, determined by $$
\int_{\R^d}\phi\,\rd(F\#\rho_0)=\int_{\R^d}\phi\circ F\,\rd\rho_0
$$
for all test functions $\phi\in \mathcal C_b(\R^d)$.
\item
For curves $t\in[0,1]\mapsto \rho_t\in \Mm$ we write $\rho\in\mathcal{C}_{w}([0,1];\Mm)$ for the continuity with respect to the narrow topology.
\item
Given a nontrivial measure $\rho\in \Mm$ we will denote by $H^1(\rd \rho)$ the Hilbert space obtained by completion of the quotient by the seminorm kernel of the space  $\mathcal{C}^1_b(\R^d)$ equipped with the Hilbert seminorm
$$
\|\phi\|_{H^1(\rd\rho)}^2=\int_{\R^d}\left(|\nabla\phi(x)|^2+|\phi(x)|^2\right) \rd \rho(x).
$$

It is not difficult to check by functional analytic tools that the Hilbert space $H^1(\rd \rho)$ can be identified with the set
\begin{multline*}
\Big\{\u=(i(\mathfrak u),j(\mathfrak u))\,|\,i(\mathfrak u)\in L^2(\rd \rho;\R), j(\mathfrak u)\in L^2(\rd \rho;\R^d),\quad \exists \{\phi^k\}\subset \mathcal{C}^1_b(\R^d),\\
\lim\limits_{k\to\infty}\phi^k = i(\mathfrak u)\ \textrm{in}\  L^2(\rd \rho;\R), \lim\limits_{k\to\infty}\nabla \phi^k = j(\mathfrak u)\ \textrm{in}\  L^2(\rd \rho;\R^d)\Big\}.
\end{multline*} 

Note that in general the elements $\mathfrak u\in H^1(\rd \rho)$ are not functions, and $j(\mathfrak u)$ is not the distributional gradient of $i(\mathfrak u)$.
As a matter of fact, neither $i(\mathfrak u)$ nor $j(\mathfrak u)$ are distributions in general, but $i(\mathfrak u)\, \rd \rho$ and $j(\mathfrak u)\, \rd \rho$ are, and only the latter ``products'' will be considered through the paper.
Nevertheless, for the sake of intuition and presentation, we will slightly abuse the notation and simply write $u$ instead of $i(\mathfrak u)$ and $\nabla u$ instead of $j(\mathfrak u)$.
In other words, one should think of elements $\u$ in $H^1(\rd\rho)$ as \emph{couples} $(u,\nabla u)$ with $u\in L^2(\rd\rho;\R)$ and $\nabla u\in L^2(\rd\rho;\R^d)$.
It is worth pointing out that $\nabla u$ does not necessarily represent the derivative of $u$ in any sense unless $\rho$ is smooth enough.
For example, when $\rho=\delta_0$ our space $H^1(\rd\delta_0)$ is isometric to $\R\times\R^d$, and the second ``gradient'' component $j(\u)=\nabla u\in L^2(\rd\delta_0;\R^d)\cong \R^d$ cannot be retrieved from the sole knowledge of $i(\u)=u\in L^2(\rd\delta_0)\cong\R$ by ``differentiating'' $u$. Alternative existing definitions of Sobolev functions with respect to measures (see, e.g., \cite{Amb15,BF03} and the references therein) are less suitable for our purposes.  
Observe that the Hilbert norm in $H^1(\rd \rho)$ coincides with
\begin{multline*}
\|\mathfrak u\|_{H^1(\rd\rho)}^2=\int_{\R^d}\left(|j(\mathfrak u)(x)|^2+|i(\mathfrak u)(x)|^2\right) \rd \rho(x)\\
=\int_{\R^d}\left(|\nabla u(x)|^2+|u(x)|^2\right) \rd \rho(x).
\end{multline*}
Note also that if $u\in L^2(\rd \rho)$ is $C^1$-smooth and bounded, then $\mathfrak u=(u,\nabla u)\in H^1(\rd \rho)$, where $\nabla$ stands now for the classical gradient. 
\item 
Given a narrowly continuous curve $\rho\in \mathcal{C}_{w}([0,1];\Mm)$ we will denote by $L^2(0,1;L^2( \rd \rho_t))$ the Hilbert space obtained by completion of the quotient by the seminorm kernel of the space $\mathcal{C}^1_b((0,1)\times\R^d)$ equipped with the Hilbert seminorm
$$
\|\phi\|_{L^2(0,1;L^2(\rd\rho_t))}^2=\int_0^1\left(\int_{\R^d}|\phi(t,x)|^2 \rd \rho_t(x)\right) \rd t.
$$
By construction, this space is isometric to the space $L^2((0,1)\times \R^d,\rd \mu)$, where the measure $\rd \mu=\rd t \otimes \rd \rho_t\in \Mm((0,1)\times\R^d)$ is defined by disintegration as
$$
\forall \phi\in\mathcal{C}_b((0,1)\times\R^d):\quad 
\iint\limits_{(0,1)\times\R^d}\phi(t,x) \rd \mu(t,x):=\int_0^1\left(\int_{\R^d}\phi(t,x) \rd \rho_t(x)\right) \rd t.
$$
Take any $u\in L^2(0,1;L^2(\rd\rho_t))$ and a sequence $\{\phi^k\}\subset \mathcal{C}^1_b((0,1)\times\R^d)$ converging to $u$ in $L^2((0,1)\times \R^d,\rd \mu)$.
Passing to a subsequence if necessary, we may assume that for a.a. $t\in (0,1)$ there exists the $L^2(\rd\rho_t)$-limit of $\phi^k_t$. The limit $u_t:=\lim\limits_{k\to\infty}\phi^k_t\in L^2(\rd\rho_t)$ 
is well-defined (does not depend on $\phi^k$) for a.a. $t$, and
$$
\|u\|^2_{L^2(0,1;L^2(\rd\rho_t))}=\int_0^1\|u_t\|^2_{L^2(\rd\rho_t)}\rd t.
$$ 

\item
Given a narrowly continuous curve $\rho\in \mathcal{C}_{w}([0,1];\Mm)$ we will denote by $L^2(0,1;H^1( \rd \rho_t))$ the Hilbert space obtained by completion of the quotient by the seminorm kernel of the space $\mathcal{C}^1_b((0,1)\times\R^d)$ equipped with the Hilbert seminorm
$$
\|\phi\|_{L^2(0,1;H^1(\rd\rho_t))}^2=\int_0^1\left(\int_{\R^d}\left(|\nabla\phi(t,x)|^2+|\phi(t,x)|^2\right) \rd \rho_t(x)\right) \rd t.
$$
As above, one can prove that $L^2(0,1;H^1( \rd \rho_t))$  can be identified with the set \begin{multline*}\Big\{\mathfrak u=(i(\mathfrak u),j(\mathfrak u))\,|\,i(\mathfrak u)\in L^2(0,1;L^2( \rd \rho_t)), j(\mathfrak u)\in L^2(0,1;L^2(\rd \rho_t;\R^d)),\\ \exists \{\phi^k\}\subset \mathcal{C}^1_b((0,1)\times\R^d), \lim\limits_{k\to\infty}\phi^k = i(\mathfrak u)\ \textrm{in}\  L^2(0,1;L^2( \rd \rho_t)), \\ \lim\limits_{k\to\infty}\nabla \phi^k = j(\mathfrak u)\ \textrm{in}\  L^2(0,1;L^2(\rd \rho_t;\R^d))\Big\}.\end{multline*} 
One can see that if $\mathfrak u\in L^2(0,1;H^1(\rd\rho_t))$ then $\mathfrak u_t:=(i(\mathfrak  u)_t,j(\mathfrak  u)_t)\in H^1(\rd\rho_t)$ is well-defined for a.e. $t\in (0,1)$, and
$$
\|\mathfrak   u\|^2_{L^2(0,1;H^1(\rd\rho_t))}=\int_0^1\|\mathfrak   u_t\|^2_{H^1(\rd\rho_t)}\rd t.
$$

In the same spirit as above, we will abuse the notation and write $u$ instead of $i(\mathfrak   u)$ and $\nabla u$ instead of $j(\mathfrak   u)$, for any $\mathfrak   u\in L^2(0,1;H^1( \rd \rho_t))$. Then the Hilbert norm in $L^2(0,1;H^1(\rd\rho_t))$ is $$
\|\mathfrak   u\|_{L^2(0,1;H^1(\rd\rho_t))}^2=\int_0^1\left(\int_{\R^d}\left(|\nabla u_t(x)|^2+|u_t(x)|^2\right) \rd \rho_t(x)\right) \rd t.
$$

\item
Given a narrowly continuous curve $\rho\in\mathcal{C}_w([0,1];\Mm)$ and $\mathfrak u\in L^2(0,1;H^1( \rd \rho_t))$ we say that
\begin{equation}
\label{eq:continuity}
\partial_t\rho_t+\dive(\rho_t\nabla u_t)=\rho_t u_t
\end{equation}
is satisfied in the distributional sense $\mathcal{D}'((0,1)\times\R^d)$ if
$$
-\int_0^1\int_{\R^d}\partial_t\phi \,\rd\rho_t(x)\rd t
=
\int_0^1\int_{\R^d}(\nabla\phi\cdot\nabla u_t+\phi u_t)\rd\rho_t(x)\rd t
$$
for all $\phi\in \mathcal{C}^\infty_c((0,1)\times\R^d)$.
We will often refer to \eqref{eq:continuity} as the \emph{non-conservative continuity equation}.
For any such $\rho\in \mathcal{C}_w([0,1];\Mm)$, an easy approximation argument shows that in fact
\begin{equation}
\label{eq:formulation_PDE_t0_t1}
\forall \phi\in \mathcal{C}^1_b(\R^d):\qquad
\int_{\R^d}\phi(\rd\rho_{t}-\rd\rho_{s})=\int_{s}^{t}\int_{\R^d}(\nabla\phi\cdot\nabla u_\tau+\phi u_\tau)\rd\rho_\tau(x) \rd\tau
\end{equation}
holds for all fixed $0\leq s\leq t\leq 1$, which is equivalent to the previous weak formulation if $t\mapsto \rho_t$ is narrowly continuous.
\item
The \emph{bounded-Lipschitz distance} between two measures $\rho_0,\rho_1\in \Mm$ is
$$
d_{BL}(\rho_0,\rho_1)=\sup\limits_{\|\phi\|_{\operatorname{Lip}}\leq 1}\left|\int_{\R^d}\phi(\rd \rho_1-\rd \rho_0)\right|.
$$
Useful properties are that the metric space $(\Mm,d_{BL})$ is complete and that $d_{BL}$ metrizes the narrow convergence on $\Mm$.
These facts are well-known \cite{dud} if we replace $\Mm$ by the space of probability measures.
The reduction of the general case to measures of unit mass is not involved. We remark that these properties may also be indirectly deduced from the results of \cite{PR13properties} and \cite{PR14generalized}. Another useful observation is that the supremum can be restricted to smooth compactly supported functions. This follows from the tightness of a set consisting of two finite Radon measures.
\item Finally, $B_R$ are open balls of radius $R>0$ in $\R^d$ centered at zero, and $C$ is a generic positive constant. 
\end{itemize}
%
\section{The metric space $(\Mm,d)$}
\label{section:metric_space}

\subsection{Definition and first properties}
\begin{defi} \label{d:metr}
Given two finite Radon measures $\rho_0,\rho_1\in \Mm(\R^d)$ we define
\begin{equation}\label{e:mini}
d^2(\rho_0,\rho_1)=\inf\limits_{\mathcal{A}(\rho_0,\rho_1)} \int_0^1\left(\int_{\R^d}(|\nabla u_t|^2+|u_t|^2)\rd\rho_t\right)\rd t,
\end{equation}
where the admissible set $\mathcal{A}(\rho_0,\rho_1)$ consists of all couples $(\rho_t,\mathfrak u_t)_{t\in [0,1]}$ such that
\begin{equation*}
\left\{ 
\begin{array}{l} 
\rho\in\mathcal{C}_w([0,1];\Mm),\\
\rho|_{t=0}=\rho_0;\quad \rho|_{t=1}=\rho_1,\\
\mathfrak u\in L^2(0,T;H^1(\rd\rho_t)),\\
\p_t\rho_t+\dive(\rho_t\nabla u_t)=\rho_tu_t \quad \mbox{in the distributional sense}.
\end{array}
\right.
\end{equation*}
\end{defi}
\noindent
As already anticipated, we shall prove shortly that
\begin{theo} 
 $d$ is a distance on $\Mm(\R^d)$.
 \label{theo:d_distance}
\end{theo}
\noindent Before going into the proof we need a preliminary result, to be used repeatedly in the sequel:
\begin{lem}[Bounded-Lipschitz and mass estimate]
Let $\rho\in \mathcal{C}_w([0,1];\Mm)$ be a narrowly continuous curve, assume that the non-conservative continuity equation $\partial_t\rho_t+\dive(\rho_t\nabla u_t)=\rho_tu_t$ is satisfied in the distributional sense for some potential $\mathfrak u\in L^2(0,T;H^1(\rd\rho_t))$ with finite energy
$$
E=E[\rho;\mathfrak u]=\|\mathfrak u\|^2_{L^2(0,T;H^1(\rd\rho_t))}=\int_0^1\int_{\R^d}(|\nabla u_t|^2+|u_t|^2)\rd\rho_t(x)\rd t,
$$
and let $M:=2(\max\{m_0,m_1\}+E)$ with $m_i=\rd\rho_i(\R^d)$.
Then the masses are bounded uniformly in time, $m_t=\rd\rho_t(\R^d)\leq M$ and
\begin{equation}
\label{eq:fundamental_dBL_estimate}
\forall\, \phi\in \mathcal{C}^1_b(\R^d):\qquad \left|\int_{\R^d}\phi(\rd\rho_t-\rd\rho_s)\right| \leq (\|\nabla \phi\|_{\infty}+\|\phi\|_{\infty})\sqrt{ME}|t-s|^{1/2}
\end{equation}
for all $0\leq s\leq t\leq 1$.
\label{lem:fundamental_dBL_estimate}
\end{lem}
\begin{proof}
By narrow continuity of $t\mapsto \rho_t$ the masses $m_t$ are uniformly bounded and $m=\max\limits_{t\in [0,t]} m_t$ is finite.
 Using the Cauchy-Schwarz inequality in \eqref{eq:formulation_PDE_t0_t1} gives
\begin{multline*}
 \left|\int_{\R^d}\phi(\rd \rho_t-\rd \rho_s)\right|
 = \left|\int_{s}^{t}\int_{\R^d}(\nabla\phi\cdot\nabla u_\tau+\phi u_\tau) \rd \rho_\tau \rd \tau\right| \\
  \leq \int_s^t\left(\int_{\R^d}(|\nabla\phi|^2+|\phi|^2)\rd\rho_\tau\right)^{1/2}\left(\int_{\R^d}(|\nabla u_\tau|^2+|u_\tau|^2)\rd\rho_\tau\right)^{1/2} \rd \tau\\
  \leq ( \|\nabla \phi\|_{\infty} + \|\phi\|_{\infty}) \sqrt{m}\cdot |t-s|^{1/2}\left(\int_0^1\int_{\R^d}(|\nabla u_\tau|^2+|u_\tau|^2)\rd\rho_\tau \rd\tau\right)^{1/2}\\
 \leq (\|\nabla \phi\|_{\infty}+\|\phi\|_{\infty})\sqrt{mE}|t-s|^{1/2},
\end{multline*}
and it is enough to estimate $m\leq M= 2(\max\{m_0,m_1\}+E)$ as in our statement.
Choosing $\phi\equiv 1$ we obtain from the previous estimate $|m_t-m_s|\leq \sqrt{mE}|t-s|^{1/2}$. Let $t_0\in [0,1]$ be any time when $m_{t_0}=m$: choosing $t=t_0$ and $s=0$ we immediately get $m\leq m_0+\sqrt{mE}|t_0-0|^{1/2}\leq m_0+\sqrt{mE}$, and some elementary algebra bounds $m\leq 2(m_0+E)$. Exchanging the roles of $\rho_0,\rho_1$ we get similarly $m\leq 2(m_1+E)$, and finally $m\leq M$.
\end{proof}
\begin{rmk}
The paths and energies can easily be scaled in time as follows: if $(\rho_t,\mathfrak u_t)_{t\in [0,1]}$ connects $\rho_0$ to $\rho_1$  with energy $E[\rho; \mathfrak u]$, then for any $T>0$ the path $(\o\rho_s,\o{ \mathfrak u}_s)=\left(\rho_{\frac{s}{T}},\frac{1}{T}\mathfrak u_{\frac{s}{T}}\right)$ connects $\rho_0$ to $\rho_1$ in time $s\in [0,T]$ with energy
\begin{align*}
E[\o\rho;\o {\mathfrak u}]
& =\int_0^T\left(\int_{\R^d}(|\nabla \o u_s|^2+|\o u_s|)\rd\o\rho_s\right)\rd s\\
& =\frac{1}{T}\int_0^1\left(\int_{\R^d}(|\nabla  u_t|^2+|u_t|)\rd\rho_t\right)\rd t=\frac{1}{T}E[\rho;\mathfrak u].
\end{align*}
\label{rmk:time_scaling}
\end{rmk}

\begin{proof}[Proof of Theorem~\ref{theo:d_distance}]
Let us first show that $d(\rho_0,\rho_1)$ is always finite for any $\rho_0,\rho_1\in \Mm$.
Indeed for any finite measure $\nu_0\in \Mm$ it is easy to see that $\nu_t=(1-t)^2\nu_0$ and $\mathfrak u_t=(u_t,\nabla u_t)=\left(-\frac{2}{1-t},0\right)$ give a narrowly continuous curve $t\mapsto \nu_t\in \Mm$ connecting $\nu_0$ to zero, and an easy computation shows that this path has finite energy $E[\nu;\mathfrak u]=4\rd\nu_0(\R^d)<\infty$ (this curve is actually the geodesic between $\nu_0$ and $0$, see later on the proof of Proposition~\ref{prop:geodesic_rho_to_0} for the details). Rescaling time as in remark~\ref{rmk:time_scaling}, it is then easy to connect any two measures $\rho_0,\rho_1\in \Mm$ in time $t\in [0,1]$ by first connecting $\rho_0$ to $0$ in time $t\in [0,1/2]$ and then connecting $0$ to $\rho_1$ in time $t\in [1/2,1]$ with cost exactly $E=2(4\rd\rho_0(\R^d)+4\rd\rho_1(\R^d))<\infty$.

In order to show that $d$ is really a distance, observe first that the symmetry $d(\rho_0,\rho_1)=d(\rho_1,\rho_0)$ is obvious by definition.

For the indiscernability, assume that $\rho_0,\rho_1\in \Mm$ are such that $d(\rho_0,\rho_1)=0$.
Let $\left(\rho^k_t,\u^k_t\right)_{t\in [0,1]}$ be any minimizing sequence in \eqref{e:mini}, i-e $\lim\limits_{k\to \infty}E[\rho^k;\u^k]=d^2(\rho_0,\rho_1)=0$.
By Lemma~\ref{lem:fundamental_dBL_estimate} we see that the masses $m^k_t=\rd\rho^k_t(\R^d)$ are uniformly bounded, $\sup\limits_{t\in [0,1],\,k\in\N}m^k_t\leq M$.
For any fixed $\phi\in \mathcal{C}^\infty_c(\R^d)$ the fundamental estimate \eqref{eq:fundamental_dBL_estimate} gives
\begin{align*} 
\left|\int_{\R^d}\phi(\rd\rho_1-\rd\rho_0)\right| \leq (\|\nabla\phi\|_{\infty}+\|\phi\|_{\infty})\sqrt{M}\cdot\sqrt{E[\rho^k;\u^k]}.
\end{align*}
Since $\lim\limits_{k\to\infty} E[\rho^k;\u^k]=d^2(\rho_0,\rho_1)=0$ we conclude that $\int_{\R^d} \phi(\rd\rho_1-\rd\rho_0)=0$ for all $\phi\in \mathcal{C}^\infty_c(\R^d)$, thus $\rho_1=\rho_0$ as desired.

As for the triangular inequality, fix any $\rho_0,\rho_1,\nu\in \Mm$ and let us prove that $d(\rho_0,\rho_1)\leq d(\rho_0,\nu)+d(\nu,\rho_1)$.
We can assume that all three distances are nonzero, otherwise the triangular inequality trivially holds by the previous point.
Let now $(\underline{\rho}^k_t,\underline{\u}^k_t)_{t\in[0,1]}$ be a minimizing sequence in the definition of $d^2(\rho_0,\nu)=\lim\limits_{k\to\infty}E[\underline{\rho}^k;\underline \u^k]$, and let similarly $(\overline{\rho}^k_t,\overline{\u}^k_t)_{t\in[0,1]}$ be such that $d^2(\nu,\rho_1)=\lim\limits_{k\to\infty}E[\overline{\rho}^k;\overline \u^k]$.
For fixed $\tau\in (0,1)$ let $\left(\rho_t,\u_t\right)$ be the continuous path obtained by first following $\left(\underline{\rho}^k,\frac{1}{\tau}\underline{\u}^k\right)$ from $\rho_0$ to $\nu$ in time $\tau$, and then following $\left(\overline{\rho}^k,\frac{1}{1-\tau}\o \u^k\right)$ from $\nu$ to $\rho_1$ in time $1-\tau$.
Then $\left(\rho^k_t,\u^k_t\right)_{t\in[0,1]}$ is an admissible path connecting $\rho_0$ to $\rho_1$, hence by definition of our distance and the explicit time scaling in Remark~\ref{rmk:time_scaling} we get that
$$
d^2(\rho_0,\rho_1)\leq E[\rho^k;\u^k]=\frac{1}{\tau}E[\o \rho^k;\o \u^k]+\frac{1}{1-\tau}E[\underline \rho^k;\underline \u^k].
$$
Letting $k\to\infty$ we obtain for any fixed $\tau\in (0,1)$
$$
d^2(\rho_0,\rho_1)\leq \frac{1}{\tau}d^2(\rho_0,\nu)+\frac{1}{1-\tau}d^2(\nu,\rho_1).
$$
Finally choosing $\tau=\frac{d(\rho_0,\nu)}{d(\rho_0,\nu)+d(\nu,\rho_1)}\in (0,1)$ yields
$$
d^2(\rho_0,\rho_1)\leq\frac{1}{\tau}d^2(\rho_0,\nu)+\frac{1}{1-\tau}d^2(\nu,\rho_1)=(d(\rho_0,\nu)+d(\nu,\rho_1))^2
$$
and the proof is complete.
\end{proof}

As an immediate consequence of Lemma~\ref{corm} we get
\begin{cor}
\label{corm} The elements of a bounded set in $(\Mm,d)$ have uniformly bounded mass.
\end{cor}
The converse statement is also true, see Corollary \ref{corm1} below. 
Another property, easily following from Remark~\ref{rmk:time_scaling}, is
\begin{lem} \label{l:hk}
 If $(\rho_t,\u_t)_{t\in [0,1]}$ is a narrowly continuous curve with total energy $E$ then $t\mapsto \rho_t$ is $1/2$-H\"older continuous w.r.t. $d$, and more precisely
 $$
 \forall\, t_0,t_1\in [0,1]:\qquad d(\rho_{t_0},\rho_{t_1})\leq \sqrt{E}|t_0-t_1|^{1/2}.
 $$
\end{lem}
\begin{proof}
Rescaling in time and connecting $\rho_{t_0}$ to $\rho_{t_1}$ by the path $(\rho_s,(t_1-t_0)\u_s)_{s\in [0,1]}$ with $t=t_0+(t_1-t_0)s$, the resulting energy scales as $d^2(\rho_{t_0},\rho_{t_1})\leq E[\o\rho;\o \u]$ $\leq E|t_0-t_1|$.
\end{proof}
Before proceeding with the study of the topological properties of our metric we need some more results.
\begin{prop}\label{prop:geodesic_rho_to_0}
For any $\rho_0\in \Mm$ there holds $d(\rho_0,0)=2\sqrt{m_0}$ with $m_0=\rd\rho_0(\R^d)$, and the geodesic is explicitly given by $\rho_t=(1-t)^2\rho_0$ and $\mathfrak u_t=(u_t,\nabla u_t)=\left(-\frac{2}{1-t},0\right)$. 
\end{prop}
Note in particular that $\nabla u_t\equiv 0$, which means that the optimal strategy to send a measure $\rho_0$ to zero is always to ``squeeze it down'' without any transport.
\begin{proof}
We start by showing that in the minimization problem in the definition of $d^2(\rho_0,0)$ one can restrict to paths of the form $\rho_t=\lambda(t)\rho_0$, and then we compute the optimal $\lambda(t)$ from which we recover $\left(\rho_t,\mathfrak u_t\right)$ with $u_t$ constant in space, i-e $\nabla u_t\equiv 0$.\\
{\it Step 1: no transport is involved}. Let $(\rho_t,\mathfrak u_t)_{t\in [0,1]}$ be any admissible path connecting $\rho_0$ to $\rho_1=0$ with finite energy. We claim that
$$
\tilde{\rho}_t:=\frac{m_t}{m_0}\rho_0,\qquad \tilde{u}_t:=\left<u_t\right>_{\rd\rho_t}=\frac{1}{m_t}\int_{\R^d}u_t \rd\rho_t,\qquad \nabla \tilde u_t\equiv 0
$$
always gives an admissible path (i-e $\mathfrak {\tilde u}=(\tilde u,\nabla\tilde u)\in L^2(0,1;H^1(\rd \rho_t))$) with lesser energy, where $m_t=\rd\rho_t(\R^d)$ denotes the mass at time $t$ as before.
Note that because the initial path $\rho_t$ is narrowly continuous we have in particular that $t\mapsto m_t$ is continuous, and we can thus assume that $m_t>0$ in $[0,1)$ (otherwise $\rho_{t_0}=0$ is attained for some time $t_0\in (0,1)$ so whatever happens after $t_0$ can only costs an unnecessary extra energy, and scaling in time $t=st_0$ with $s\in (0,1)$ decreases the total cost).
This continuity also shows that $\tilde{\rho}_t$ connects $\rho_0$ to $0$, since $m_t\to 0$ implies narrow convergence $\tilde\rho_t\to 0$ when $t\to 1$.
Since $\tilde{u}_t=\left<u_t\right>_{\rd\rho_t}$ is constant in space and $\nabla\tilde u_t\equiv 0$ we have by Jensen's inequality
\begin{multline*}
E[\tilde{\rho};\tilde{\mathfrak u}] 
  =\int_0^1 \left((0+\left<u_t\right>^2_{\rd\rho_t})m_t\right)\rd t \\
  \leq \int_0^1\left<|u_t|^2\right>_{\rd\rho_t}m_t\,\rd t
  =  \int_0^1 \left(\int_{\R^d}|u_t|^2\rd\rho_t\right)\rd t\leq E[\rho;\mathfrak u],
\end{multline*}
which shows that $(\tilde\rho_t,\tilde{\mathfrak u}_t)$ has lesser energy as claimed.

It remains to show that this path solves the non-conservative continuity equation.
Taking $\phi(x)\equiv 1$ in the weak formulation of the non-conservative continuity equation satisfied by $(\rho_t,\mathfrak u_t)$ and exploiting $\sup_{t}m_t\leq M$ it is easy to check that $\frac{d}{dt}m_t=\int_{\R^d}u_t \rd \rho_t\in L^2(0,1)$. Then in the sense of distributions $\mathcal{D}'((0,1)\times \R^d)$ we have
\begin{align*}
\p_t\tilde{\rho}_t
 = \p_t\left(\frac{m_t}{m_0}\rho_0\right)& =\frac{d m_t}{dt}\cdot\frac{\rho_0}{m_0} =\left(\int_{\R^d}u_t \rd\rho_t\right)\frac{\rho_0}{m_0}\\
& =\left(\frac{1}{m_t}\int_{\R^d}u_t \rd\rho_t\right)\frac{m_t}{m_0}\rho_0 =\tilde{u}_t\tilde{\rho}_t,
\end{align*}
and we conclude recalling that by construction $\nabla\tilde u_t\equiv 0$ in the advection term.\\
{\it Step 2: computing the geodesic.} By the previous step it is enough to minimize over all $(\rho_t,\mathfrak u_t)$ such that $\rho_t=\lambda(t)\rho_0$ with $\lambda(t)>0$ in $[0,1)$, and constant-in-space potentials $\nabla u_t\equiv 0$. For such paths it is easy to realize that necessarily $u_t=\frac{\lambda'(t)}{\lambda(t)}$, thus we only have to solve the minimization problem
$$
\inf\left\{m_0\int_0^1\left|\frac{\lambda '(t)}{\lambda(t)}\right|^2\lambda(t)\rd t:\qquad \lambda(0)=1,\,\lambda(1)=0\right\},
$$
where $m_0=\rd\rho_0(\R^d)$.
It is a simple exercise in the calculus of variations to check that the unique minimizer is
$$
\lambda(t)=(1-t)^2,\qquad u(t)=\frac{\lambda'(t)}{\lambda(t)}=-\frac{2}{1-t},\qquad \rho_t=\lambda(t)\rho_0=(1-t)^2\rho_0
$$
and the explicit computation finally gives
$$
d^2(\rho_0,0)
= m_0\int_0^1\left|u(t)\right|^2\lambda(t)\rd t = m_0\int_0^1\left|\frac{2}{1-t}\right|^2(1-t)^2\rd t = 4 m_0
$$
as claimed. 
\end{proof}
As an easy consequence of the previous Proposition~\ref{prop:geodesic_rho_to_0} we obtain
\begin{cor}
\label{corm1} Subsets of $\Mm$ with uniformly bounded mass are bounded in $(\Mm,d)$ as $\rd\rho(\R^d)\leq M\Rightarrow d(\rho,0)\leq 2\sqrt{M}$.
\end{cor}
\begin{prop}[Scaling properties]
\label{prop:scaling_lambda}
For any $\lambda\in\R^+$ and measures\linebreak $\rho_0,\rho_1\in \Mm$ we have
\begin{equation}
\label{eq:scaling_rho_lambda.rho}
d(\rho_0,\lambda \rho_0)=2\sqrt{m_0}\left|1-\sqrt{\lambda}\right|
\end{equation}
and
\begin{equation}
\label{eq:scaling_lambda.rho0_lambda.rho1}
d(\lambda\rho_0,\lambda\rho_1)=\sqrt{\lambda}d(\rho_0,\rho_1).
\end{equation}
\end{prop}
\noindent Remark that \eqref{eq:scaling_rho_lambda.rho} can be rephrased more intrinsically as: ``if $\rho_0,\rho_1\in \Mm$ are proportional measures then $d(\rho_0,\rho_1)=2|\sqrt{m_0}-\sqrt{m_1}|$'', and also that the previous Proposition~\ref{prop:geodesic_rho_to_0} is a particular case with $\lambda=0$.
\begin{proof}
For the first part it is easy to argue as in the proof of Proposition~\ref{prop:geodesic_rho_to_0} to see that it is enough to minimize over all paths $\rho_t=\lambda(t)\rho_0$, with $u_t=\frac{\lambda'(t)}{\lambda(t)},\nabla u_t\equiv 0$ (i-e averaging in space decreases the energy) and the constraints $\lambda(0)=1$, $\lambda(1)=\lambda$.
Finding the optimal $\lambda(t)$ is again a simple exercise in the calculus of variations, and the explicit computation leads to \eqref{eq:scaling_rho_lambda.rho} (the minimizer $\lambda(t)$ is again a second order polynomial). 

For \eqref{eq:scaling_lambda.rho0_lambda.rho1}, note that our statement is trivial if $\lambda=0$ so we can assume that $\lambda>0$. We denote below $\mu_0=\lambda\rho_0,\mu_1=\lambda\rho_1$. 
Let $\left(\rho^k_t,\u^k_t\right)_{t\in [0,1]}$ be a minimizing sequence in the definition of $d^2(\rho_0,\rho_1)$. Because both the non-conservative continuity equation and the energy are linear in $\rho$ we see that $(\lambda \rho^k_t,\u^k_t)$ is an admissible path connecting $\mu_0$ to $\mu_1$, and
$$
d^2(\mu_0,\mu_1)\leq E[\lambda \rho^k;\u^k]=\lambda E[\rho^k;\u^k]\underset{k\to\infty}{\rightarrow} \lambda d^2(\rho_0,\rho_1).
$$
The other inequality is obtained similarly: if $(\mu^k_t,\mathfrak v^k_t)$ is a minimizing sequence for $d^2(\mu_0,\mu_1)$ then $(\frac{1}{\lambda}\mu^k_t,\mathfrak v^k_t)$ connects $\rho_0$ to $\rho_1$, thus
$$
d^2(\rho_0,\rho_1)\leq E[\mu^k/\lambda;\mathfrak v^k]=\frac{1}{\lambda}E[\mu^k;\mathfrak v^k]\underset{k\to\infty}{\rightarrow}\frac{1}{\lambda}d^2(\mu_0,\mu_1)
$$
and the proof is complete.
\end{proof}
\subsection{Topological properties}
\begin{theo} \label{dens}
The compactly supported measures are dense in $(\Mm,d)$: for any $\rho\in \Mm$ and $\eps>0$ there exists $\rho'\in \Mm$ compactly supported such that $d(\rho,\rho')\leq \eps$.
\end{theo}
\begin{proof}
Observe that $\rho$ has arbitrarily small mass outside of $B_R$ for large $R$. The argument goes in two steps: first we create an annular gap around $|x|=R$ with arbitrarily small cost, i-e construct a measure $\tilde{\rho}$ which has support in $B_R\cup (\R^d\setminus B_{R+\delta})$ for some small $\delta>0$ such that $d(\rho,\tilde{\rho})\leq \eps/2$ and the mass of $\tilde{\rho}$ outside of $B_{R+\delta}$ is still small. The second step consists in sending all the exterior mass $\rd\tilde\rho(\R^d\setminus B_{R+\delta})$ to zero while keeping the interior part $\tilde\rho|_{B_R}$ unchanged (the fact that we can do both simultaneously relies on the gap of size $\delta>0$). In fact we will do so without modifying the original measure $\rho$ inside $B_R$, so that really in the end we will take $\rho '=\rho|_{B_R}$ with $d(\rho,\rho')\leq \eps$.

\noindent{\it Step 1: creating the gap}. For fixed $R>0$ let us first decompose
$$
\rho=\underline{\rho}+\overline{\rho}\qquad \mbox{with }\underline{\rho}:=\rho\mathbf{1}_{B_R}\mbox{ and }\overline{\rho}:=\rho\mathbf{1}_{\R^d\setminus B_R}.
$$
Let also $U(x)=U_R(x)\in \mathcal{C}^{\infty}(\R^d)$ be any smooth function such that
$$
U(x)=\left\{
\begin{array}{ll}
|x|-R	& \mbox{if }R\leq  |x|\leq R+1\\
2 & \mbox{if }|x|>R+2
\end{array}
\right.
\quad \mbox{and}\quad |U|,|\nabla U|\leq 2.
$$
By \cite[Prop. 3.6]{maniglia2007probabilistic} we can solve
$$
\left\{
\begin{array}{l}
\partial_{t}\overline{\rho}_t+\dive(\o{\rho}_t\nabla U)=\o{\rho}_tU\\
\o{\rho}|_{t=0}=\o\rho,
\end{array}
\right. 
$$
and $t\mapsto \o\rho_t$ is narrowly continuous. Moreover by construction the initial datum $\o\rho$ has support in $\R^d\setminus B_R$, so for $t\in[0,1]$ it is easy to see that the measure $\o\rho_t$ has support in $\R^d\setminus B_{R+t}$ (the characteristics $\frac{dx}{dt}=\nabla U(x)$ diverge radially away from $B_R$, with constant speed $|\nabla U|=1$ for $|x|\gtrsim R$). Denoting the mass
$$
\o m_t:=\int_{\R^d}\rd\o\rho_t=\int_{\R^d\setminus B_{R}}\rd\o\rho_t=\int_{\R^d\setminus B_{R+t}}\rd\o\rho_t
$$
it is easy to check from the non-conservation continuity equation that
$$
\left|\frac{d }{dt}\o m_t\right|=\left|\int_{\R^d}U\rd\o\rho_t\right|\leq \int_{\R^d}|U|\rd\o\rho_t\leq 2 \o m_t,
$$
so that
$$
\o m_t \leq e^{2t}\o m_0
$$
for $t\in[0,1]$. Define now
$$
\rho_t:=\underline{\rho}+\o \rho_t
\qquad\mbox{and}\qquad
u_t(x)=\left\{
\begin{array}{ll}
0 & \mbox{if }|x|< R\\
U(x) & \mbox{if } |x|\geq R
\end{array}
\right.
$$
(observe that $u_t$ is constant in time, Lipschitz-continuous in space, and uniformly bounded by $|u_t(x)|\leq 2$). By construction for $t>0$ the measure $\rho_t$ splits into a measure $\underline{\rho}$ with support in $B_R$ and a measure $\o\rho_t$ with support in $\R^d\setminus B_{R+t}$.
Exploiting the fact that these supports stay at distance $t>0$ away from each other it is easy to check that $(\rho_t,\u_t)$ solves the non-conservative continuity equation in the sense of distributions (for $t>0$ the ``matching'' at $|x|=R$ is never seen!).
The positive gap moreover allows to compute for a.e. $t\in(0,1)$
\begin{multline*}
\int_{\R^d}(|\nabla u_t|^2+|u_t|^2)\rd\rho_t
 = \int_{B_R}(|\nabla u_t|^2+|u_t|^2)\rd\underline{\rho}+ \int_{ B_{R+t}^{\complement}}(|\nabla u_t|^2+|u_t|^2)\rd\o{\rho}_t\\
 \leq 0 +(\|\nabla U\|^2_{\infty}+\|U\|^2_{\infty})\int_{\R^d\setminus B_{R+t}}\rd\o{\rho}_t
 \leq C\o m_t \leq Ce^{2t}\o m_0,
\end{multline*}
which shows that this path has finite energy. In particular by Lemma~\ref{l:hk} $t\mapsto\rho_t$ is continuous with respect to our metric $d$, and $d(\rho,\rho_t)$ is small for small $t>0$.
\\
Now for fixed $\eps>0$ there exists $R>0$ such that $\o m_0=\rd\rho(\R^d\setminus B_R)\leq \eps$.
Choosing $t=\delta>0$ small enough we get that the measure $\tilde{\rho}:=\o\rho_{t=\delta}$ satisfies $\operatorname{supp}(\tilde{\rho})\subset B_R\cup(\R^d\setminus B_{R+\delta})$,
$$
\tilde{\rho}|_{B_R}=\rho|_{B_R},
\qquad
\rd\tilde{\rho}(\R^d\setminus B_{R+\delta})=\o m_{\delta}\leq e^{2\delta}\o m_0\leq 2\eps,
$$
and
$$
d(\rho,\tilde{\rho})\leq \eps.
$$
\noindent
{\it Step 2: sending the exterior mass to zero.} Choose now any smooth function $\tilde{u}(x)$ such that
$$
\tilde{u}(x)=
\left\{
\begin{array}{ll}
 0 & \mbox{if }|x|\leq R\\
 1 & \mbox{if }|x|\geq R+\delta
\end{array}
\right.
$$
and let
$$
\tilde{u}_t(x):=-\frac{2}{1-t}\tilde{u}(x)
\quad\mbox{and}\quad
\tilde\rho_t:=\tilde{\rho}\mathbf{1}_{B_R}+(1-t)^2\tilde{\rho}\mathbf{1}_{\R^d\setminus B_{R+\delta}}.
$$
Since $\tilde{\rho}_t$ always has a gap between $|x|< R$ and $|x|\geq R+\delta$ it is straightforward to check that $(\tilde\rho_t,\tilde{\u}_t)$ is an admissible curve connecting $\tilde{\rho}$ to $\rho'=\rho\mathbf{1}_{B_R}=\tilde{\rho}\mathbf{1}_{B_R}$ (in particular $\tilde{\rho}_t$ is narrowly continuous and solves the non-conservative continuity equation), with cost exactly
\begin{align*}
\int_0^1\left(\int_{\R^d}(|\nabla\tilde u_t|^2+|\tilde u_t|^2)\rd\tilde\rho_t\right)\rd t
& =0+\int_0^1\left(\int_{\R^d\setminus B_{R+\delta}}(|\nabla\tilde u_t|^2+|\tilde u_t|^2)\rd\tilde\rho_t\right)\rd t\\
& =4\o m_{\delta}\leq 8\eps.
\end{align*}
Indeed $\tilde{u}_t\equiv 0$ in $\overline B_{R}$, the measure $\tilde{\rho}_t$ does not charge the annulus $R<|x|<R+\delta$, and by construction outside of $B_{R+\delta}$ the potential $\tilde{u}_t$ it is exactly the geodesic between $\rho'$ to zero, see the proof of Proposition~\ref{prop:geodesic_rho_to_0}.
By the triangular inequality we finally get
$$
d(\rho,\rho')\leq d(\rho,\tilde{\rho})+d(\tilde\rho,\rho')\leq \eps+\sqrt{8\eps},
$$
and, since $\eps>0$ was arbitrary and $\rho'$ is compactly supported, the conclusion follows.
\end{proof}
\begin{prop}
\label{prop:d<W2}
Let $\mathcal{P}_2(\R^d)\subset \Mm_b(\R^d)$ be the set of Borel probability measures with finite second moments and $\mathcal{W}_2$ the quadratic Kantorovich-Rubinstein-Wasserstein distance defined on $\mathcal{P}_2$. Then for any $\rho_0,\rho_1\in \mathcal{P}_2$ there holds
\begin{equation*}\label{e:vasour}
d(\rho_0,\rho_1)\leq \mathcal{W}_2(\rho_0,\rho_1).
\end{equation*}
\end{prop}
\begin{proof}
It is well known \cite{villani03topics,villani08oldnew} that $\left(\mathcal{P}_2(\R^d),\mathcal{W}_2\right)$ is a geodesic space, so we can connect $\rho_0$ and $\rho_1$ by a constant-speed geodesic path $\rho_t$ in this Wasserstein space. By \cite[Theorem 13.8]{villani08oldnew} there exists a velocity field $\mathbf v_t\in L^\infty(0,1;L^2(\rd\rho_t))$ such that
\begin{equation}
\label{e:ce2}
\p_t\rho_t+\dive(\rho_t\mathbf  v_t)=0
\end{equation}
holds in the distributional sense, and $t\mapsto\rho_t$ is narrowly continuous (since it is continuous with values in the stronger Wasserstein metric topology).

For almost every $t\in (0,1)$ the distribution $\zeta_t=-\dive(\rho_t \mathbf v_t)\in \mathcal{D}'(\R^d)$ defines by duality a continuous linear form on $H^1(\rd\rho_t)$ with norm $\|\zeta_t\|_{H^{-1}(\rd\rho_t)}\leq \|\mathbf v_t\|_{L^2(\rd\rho_t)}$, thus by the Riesz representation theorem  we can find a unique $\u_t=(u_t,\nabla u_t)\in H^1(\rd\rho_t)$ such that
$$
\left(\u_t,.\right)_{H^1(\rd\rho_t)}=\zeta_t
\quad \mbox{and}\quad
\|\u_t\|_{H^1(\rd\rho_t)}=\|\zeta_t\|_{H^{-1}(\rd\rho_t)}\leq \|\mathbf v_t\|_{L^2(\rd\rho_t)}.
$$
In particular by definition of $(.,.)_{H^1(\rd\rho_t)}$ for a.e. $t\in (0,1)$ there holds
$$
-\dive(\rho_t\nabla u_t)+\rho_t u_t=-\dive(\rho_t \mathbf v_t)\qquad\mbox{in }\mathcal{D}'(\R^d),
$$
and by \eqref{e:ce2} it is easy to check that $\partial_t \rho_t +\dive(\rho_t\nabla u_t)=\rho_t u_t$ in the sense of distributions.
Since $t\mapsto\rho_t$ is narrowly continuous we see that $(\rho_t,\u_t)$ is an admissible path in the definition of our distance, and in particular
$$
d^2(\rho_0,\rho_1)\leq E[\rho;\u]=\int_0^1\|\u_t\|^2_{H^1(\rd\rho_t)}\rd t\leq \int _0^1\|\mathbf v_t\|^2_{L^2(\rd\rho_t)}\rd t.
$$  
By the Benamou-Brenier formula \cite{AGS06,BenamouBrenier00}, the right-hand side coincides with the squared Wasserstein distance $\mathcal{W}_2^2(\rho_0, \rho_1)$ and the proof is achieved.
\end{proof}

\begin{theo}
\label{theo:equiv}
The metric $d$ is topologically equivalent to the bounded-Lipschitz metric $d_{BL}$, i-e $d(\rho^k,\rho)\to 0$ if and only if $d_{BL}(\rho^k,\rho)\to 0$. Moreover, $(\Mm,d)$ is a complete metric space, $d$ metrizes the narrow convergence of measures on $\Mm(\R^d)$, 
and for any $\rho_0,\rho_1\in \Mm$ with masses $m_0,m_1$ there holds
\begin{equation}
 \label{eq:dBL_loc_equivalent_d}
 d_{BL}(\rho_0,\rho_1)\leq 6\sqrt{m_0+m_1}d(\rho_0,\rho_1)
\end{equation}
\end{theo}
\begin{proof} [Proof of Theorem~\ref{theo:equiv}]
We proceed in two steps: in step 1 we first prove that convergence in $d_{BL}$ implies convergence in $d$, and then in step 2 we establish \eqref{eq:dBL_loc_equivalent_d} and use this to deduce that any Cauchy sequence for $d$ is Cauchy for $d_{BL}$. Recalling that $(\Mm,d_{BL})$ is complete, we see by step 2 that any Cauchy sequence in $(\Mm,d)$ is Cauchy in $(\Mm,d_{BL})$ and therefore converges in $(\Mm,d_{BL})$, so by step 1 it converges in $(\Mm,d)$ as well. Also, since any converging sequence is Cauchy this immediately implies that any $d$-converging sequence is also $d_{BL}$-converging. Finally, since $d_{BL}$ metrizes the narrow convergence clearly so does $d$.\\
For convenience we split the first step into 1a and 1b: the former essentially reduces the problem to compactly supported (sequences of) measures, and the latter concludes by renormalizing to unit masses and comparing $d$ with the Wasserstein distance $\mathcal{W}_2$.

\emph{Step 1a.}
Take any converging sequence $\rho^k\to\rho$ in $(\Mm,d_{BL})$. 
Because the bounded-Lipschitz distance metrizes the narrow convergence the masses converge $m^k \to m$, and the sequence $\{\rho^k\}$ is tight. We want to prove that it converges in $(\Mm,d)$.\\
Since only a countable number of concentric spheres in $\R^d$ can have a positive Radon measure, there exists a sequence $R_n\to \infty$ such that $\rd\rho(\partial B_{R_n})=0$. By \cite[Prop. 1.203]{fonseca_leoni_07_modern}, $\rd\rho(A)\leq \underset{k \to \infty}{\liminf\,}\rd \rho^k(A)$ for any open set $A\subset \R^d$, and $\rd\rho(B_{R_n})= \underset{k \to \infty}{\lim\,} \rd\rho^k(B_{R_n})$. Hence, $\rd\rho|_{B_{R_n}}(A)\leq \underset{k \to \infty}{\liminf\,} \rd\rho^k|_{B_{R_n}}(A)$ for any open $A\subset \R^d$, and $\rd\rho|_{B_{R_n}}({\R^d})= \underset{k \to \infty}{\lim\,} \rd\rho^k|_{B_{R_n}}({\R^d})$. 
Consequently, by \cite[Prop. 1.206]{fonseca_leoni_07_modern},
\begin{equation}
\label{eq:rho_k_CV_narrow}
\forall \,n\in\N:\qquad\rho^k|_{B_{R_n}}\underset{k \to \infty}{\longrightarrow} \rho|_{B_{R_n}} \mbox{ narrowly}. 
\end{equation}
Owing to the tightness of $\{\rho^k\}_{k\in \N}$ and using the same construction as in the proof of Theorem \ref{dens} one easily deduces that
$$
d(\rho|_{B_{R_n}},\rho)+\sup\limits_k d(\rho^k|_{B_{R_n}},\rho^k)\underset{n \to \infty}{\longrightarrow} 0,
$$
thus by triangular inequality
$$
d(\rho^k,\rho)\leq d(\rho^k,\rho^k|_{B_{R_n}})+d(\rho^k|_{B_{R_n}},\rho|_{B_{R_n}})+d(\rho|_{B_{R_n}},\rho)
$$
it suffices to prove that
\begin{equation}
\label{eq:CV_rho^k_rho_restriction_Bn}
\forall  \,n\in \N:\qquad \rho^k|_{B_{R_n}}\underset{k \to \infty}{\longrightarrow} \rho|_{B_{R_n}}\mbox{ in }(\Mm,d).
\end{equation}

{\it Step 1b.} From now on we argue for fixed $n$, and denote for simplicity $\rho^k_n=\rho^k|_{B_{R_n}}$ and $\rho_n=\rho|_{B_{R_n}}$.
Letting $m^k_n=\rd\rho^k(B_{R_n})$ and $m_n=\rd\rho(B_{R_n})$ we see by \eqref{eq:rho_k_CV_narrow} that $m^k_n\to m_n$ when $k\to \infty$. If the limit $\rd\rho_n(\R^d)=m_n=0$ then $\rho_n=0$ in $\Mm$, so by Proposition~\ref{prop:geodesic_rho_to_0} we have $d(\rho^k_n,\rho_n)=d(\rho^k_n,0)=2\sqrt{m^k_n}\to 0$ and \eqref{eq:CV_rho^k_rho_restriction_Bn} clearly holds. Otherwise using $m_k^n\to m_n>0$ we have by the scaling Proposition~\ref{prop:scaling_lambda}
\begin{align*}
d(\rho^k_n,\rho_n)
&	\leq d\left(\rho^k_n,\frac{m_n}{m^k_n}\rho^k_n\right)+d\left(\frac{m_n}{m^k_n}\rho^k_n,\rho_n\right)\\
& = 2\left|\sqrt{m^k_n}-\sqrt{m_n}\right|+\sqrt{m_n}d\left(\frac{\rho^k_n}{m^k_n},\frac{\rho_n}{m_n}\right),
\end{align*}
and it suffices to prove that the sequence of renormalized \emph{probability} measures $\o\rho^k_n=\frac{\rho^k_n}{m^k_n}$ converges to $\o\rho_n=\frac{\rho_n}{m_n}$. Note that by construction $\o\rho^k_n,\o\rho_n$ are supported in a fixed ball $B_{R_n}$, so in particular they have uniformly bounded second moments. Applying Proposition~\ref{prop:d<W2} we see that $d(\o\rho^k_n,\o\rho_n)\leq \mathcal{W}_2(\o\rho^k_n,\o\rho_n)$, and recalling that $m^k_n\to m_n>0$ and that $\rho^k_n\to \rho_n$ narrowly it is easy to see that $\o\rho^k_n\to \o\rho_n$ narrowly as well. Applying \cite[Thm. 7.12]{villani03topics} we conclude that $\mathcal{W}_2(\o\rho^k_n,\o\rho_n)\to 0$, whence $d(\o\rho^k_n,\o\rho_n)\to 0$ and \eqref{eq:CV_rho^k_rho_restriction_Bn} holds as desired.

\emph{Step 2}.
Fix $\rho_0,\rho_1$, and let $(\rho_t,\mathfrak u_t)$ be any admissible path from $\rho_0$ to $\rho_1$ with finite energy $E$.
Taking the supremum over $\phi$ in $\eqref{eq:fundamental_dBL_estimate}$ we get $d_{BL}(\rho_0,\rho_1)\leq \sqrt{ME}$,
where $M=2(\max\{m_0,m_1\}+E)$ as in Lemma~\ref{lem:fundamental_dBL_estimate}.
Choosing now a minimizing sequence instead of an arbitrary path and taking the limit we essentially obtain the same estimate with $E=\lim E[\rho^k;\mathfrak u^k]=d^2(\rho_0,\rho_1)$, whence
$$
d_{BL}(\rho_0,\rho_1)\leq \sqrt{2(\max\{m_0,m_1\}+d^2(\rho_0,\rho_1))}d(\rho_0,\rho_1).
$$
By the triangular inequality and Proposition~\ref{prop:geodesic_rho_to_0} we control $d^2(\rho_0,\rho_1)\leq 4(d^2(\rho_0,0)+d^2(0,\rho_1))=16(m_0+m_1)$, which immediately yields \eqref{eq:dBL_loc_equivalent_d}.

Finally, let $\rho^k$ be a Cauchy sequence in $(\Mm,d)$ with mass $m^k=\rd\rho^k(\R^d)$. Since Cauchy sequences are bounded we control $4 m^k=d^2(\rho^k,0)\leq C$ uniformly in $k$, thus from \eqref{eq:dBL_loc_equivalent_d} we see that
$$
d_{BL}(\rho^p,\rho^q)\leq 6\sqrt{m^p+m^q}d(\rho^p,\rho^q)\leq Cd(\rho^p,\rho^q).
$$
As a consequence $\rho^k$ is Cauchy for the bounded-Lipschitz distance and the proof is achieved.

\end{proof}

\begin{rmk}
\label{rmk:ll}
Observe that we did not prove that the distances $d$ and $d_{BL}$ are \emph{Lipschitz equivalent}, in the sense that the identity map $\operatorname{id}:(\Mm,d)\to (\Mm,d_{BL})$ may not be globally bi-Lipschitz.
This is actually impossible due to the different scalings: From Proposition~\ref{prop:scaling_lambda} we know that $d(\lambda\rho_0,\lambda\rho_1)=\sqrt\lambda d(\rho_0,\rho_1)$, but it is easy to check that $d_{BL}(\lambda\rho_0,\lambda\rho_1)=\lambda d(\rho_0,\rho_1)$.
However, our estimate \eqref{eq:dBL_loc_equivalent_d} shows that $\operatorname{id}:(\Mm,d)\to (\Mm,d_{BL})$ is Lipschitz on bounded sets.
\end{rmk}
%


\subsection{Characterization of Lipschitz curves}
\begin{theo}\label{theo:lipschitz_curves}
Let $\{\rho_t\}_{t\in [0,1]}$ be a $L$-Lipschitz curve w.r.t. our metric $d$. Then there exists a potential $\u\in L^2(0,1;H^1(\rd\rho_t))$ such that
$$
\p_t\rho_t+\dive(\rho_t\nabla u_t)=\rho_tu_t\qquad \mbox{in }\mathcal{D}'((0,1)\times\R^d),
$$
and
$$
\|\u_t\|_{H^1(\rd\rho_t)}\leq L\qquad \mbox{a.e. }t\in (0,1).
$$
Conversely if $t\mapsto \rho_t\in \Mm$ is a narrowly continuous and $(\rho_t,\u_t)_{t\in [0,1]}$ solves the non-conservative continuity equation with $\|\u_t\|_{H^1(\rd\rho_t)}\leq L$ for a.e. $t\in (0,1)$, then $t\mapsto\rho_t$ is $L$-Lipschitz with respect to the distance $d$.
\end{theo}
\noindent

Before proceeding with the proof we will need the following technical lemma:
\begin{lem}\label{lem:eq:dist_approx_geodesic}
For any $\rho_0,\rho_1\in \Mm$ and $\eps>0$ there exists a narrowly continuous curve $\rho_t\in \mathcal{C}_w([0,1],\Mm)$ connecting $\rho_0$ to $\rho_1$ and a potential $ \u\in L^2(0,1;H^1(\rd\rho_t))$, both depending on $\eps$, solving the non-conservative continuity equation and such that
\begin{equation}
\label{eq:dist_approx_geodesic_1}
\forall \tau\in [0,1]:\qquad d(\rho_0,\rho_\tau)\leq (1+\eps)d(\rho_0,\rho_1)
\end{equation}
with also
\begin{equation}
\label{eq:eq:dist_approx_geodesic_2}
d^2(\rho_0,\rho_1)\leq E[\rho; \u]=\int_0^1\left(\int_{\R^d}(|\nabla  u_t|^2+| u_t|^2)\rd \rho_t\right)\rd t\leq (1+\eps)^2d^2( \rho_0, \rho_1).
\end{equation}
\end{lem}
\begin{proof}
If $\{\rho^k_t, \u^k_t\}_{k\in \N}$ is any minimizing sequence in the definition of $d^2(\rho_0,\rho_1)$ then \eqref{eq:eq:dist_approx_geodesic_2} is obviously satisfied for large $k\geq k_0$ and we only have to check that \eqref{eq:dist_approx_geodesic_1} holds as well if $k$ is large enough.
Note that for any $k\in \N$ and fixed $\tau\in [0,1]$, a simple time reparametrization $s=\tau t$ gives an admissible curve $(\o\rho_s,\o \u_{s})_{s\in [0,1]}:=(\rho^k_{\tau s},\tau \u^k_{\tau s})_{s\in [0,1]}$ connecting $\rho_0$ to $\rho_{\tau}$ in time $s\in [0,1]$. By definition of our distance, changing variables, and because $\tau\leq 1$, we get
\begin{multline*}
d^2(\rho_0,\rho_\tau)
 \leq \int_0^1\left(\int_{\R^d}(|\nabla \o u_s|^2+|\o u_s|^2)\rd \o\rho_s\right)\rd s\\
 = \tau\int_0^\tau\left(\int_{\R^d}(|\nabla  u_t^k|^2+| u_t^k|^2)\rd\rho_t^k\right)\rd t\\
 \leq \int_0^1\left(\int_{\R^d}(|\nabla u^k_t|^2+| u^k_t|^2)\rd\rho_t\right)\rd t.
\end{multline*}
Then for $k$ large enough the last term above converges to $d^2(\rho_0,\rho_1)$ and the conclusion follows.
\end{proof}
\begin{proof}[Proof of Theorem~\ref{theo:lipschitz_curves}]
The argument is similar to \cite[Thm. 13.8]{villani08oldnew} and \cite[Thm. 8.3.1]{AGS06}.
Fix any $\psi\in \mathcal{C}^\infty_c(\R^d)$. Since we assumed that
$$
\forall \,t,s\in [0,1]:\qquad d(\rho_t,\rho_s)\leq L|t-s|
$$
and because locally our distance is Lipschitz-stronger than the bounded-Lipschitz one (see Remark~\ref{rmk:ll}), we see that
$$
t\mapsto \Psi(t):=\int_{\R^d}\psi \,\rd\rho_t
$$
is locally Lipschitz thus differentiable almost everywhere. Fix any time $t_0\in [0,1]$ where $\Psi$ is differentiable, and let $h_n\to 0$ with $t_0+h_n\in [0,1]$.
Taking $\eps=\frac{1}{n}$, $\tilde\rho_0=\rho_{t_0}$, and $\tilde\rho_1=\rho_{t_0+h_n}$, let $(\tilde{\rho}^n_t,\tilde \u^n_t)_{t\in [0,1]}$ be any curve from Lemma~\ref{lem:eq:dist_approx_geodesic} connecting $\rho_{t_0}$ to $\rho_{t_0+h_n}$ in time $t\in [0,1]$ while solving the non-conservative continuity equation and satisfying \eqref{eq:dist_approx_geodesic_1}\eqref{eq:eq:dist_approx_geodesic_2}. Using the Cauchy-Schwarz inequality (first in space and then in time) gives that
\begin{multline*}
\left|\frac{\Psi(t_0+h_n)-\Psi(t_0)}{h_n}\right|
 =\left|\frac{1}{h_n}\int_{\R^d}\psi(\rd\rho_{t_0+h_n}-\rd\rho_{t_0})\right|\\
  =\left|\frac{1}{h_n}\int_0^1\left(\int_{\R^d}(\nabla \psi\cdot\nabla \tilde u_t^n+\psi \tilde u_t^n) \rd \tilde\rho_t^n\right) \rd t\right|\\
  \leq \underbrace{\left(\int_0^1\left(\int_{\R^d}(|\nabla \psi|^2 +|\psi|^2) \rd \tilde\rho_t^n\right) \rd t\right)^{\frac{1}{2}}}_{:=A_n}\\
  \times
\underbrace{\frac{1}{h_n}\left(\int_0^1\left(\int_{\R^d}(|\nabla \tilde u^n_t|^2 +|\tilde u^n_t|^2) \rd \tilde\rho_t^n\right) \rd t\right)^{\frac{1}{2}}}_{:=B_n}.
\end{multline*}
From the previous Lemma~\ref{lem:eq:dist_approx_geodesic} and the definition of $\tilde\rho^n$ we first estimate, with $\eps=\frac{1}{n}$ in \eqref{eq:eq:dist_approx_geodesic_2},
$$
\limsup\limits_{n\to\infty} B_n\leq \limsup\limits_{n\to\infty}(1+1/n)\frac{d(\rho_{t_0},\rho_{t_0+h_n})}{h_n}\leq L.
$$
In order to handle the term $A_n$, let us define by duality the measures $\mu^n\in \Mm((0,1)\times\R^d)$ as
$$
\forall \phi\in\mathcal{C}_b((0,1)\times\R^d):\quad
\iint\limits_{(0,1)\times\R^d}\phi(t,x) \rd \mu^n(t,x):=\int_0^1\left(\int_{\R^d}\phi(t,x) \rd \tilde\rho^n_t(x)\right) \rd t.
$$
Using again Lemma~\ref{lem:eq:dist_approx_geodesic}, and in particular \eqref{eq:dist_approx_geodesic_1} with $\eps=1/n$, we see that for all fixed $t\in [0,1]$ there holds
$$
d(\rho_{t_0},\tilde{\rho}^n_t)\leq (1+1/n)d(\rho_{t_0},\rho_{t_0+h_n})\to 0
$$
when $n\to\infty$ (because the path $\rho_t$ is Lipschitz). By Theorem~\ref{theo:equiv} we get
$$
\forall\,t\in[0,1]:\qquad \tilde\rho^n_t\to \rho_{t_0}\quad\mbox{ narrowly}.
$$
A simple application of Lebesgue's dominated convergence shows that $\mu^n$ converges narrowly to the measure $\mu$ similarly defined by
$$
\forall \phi\in\mathcal{C}_b((0,1)\times\R^d):\quad
\iint\limits_{(0,1)\times\R^d}\phi(t,x) \rd \mu(t,x):=\int_0^1\left(\int_{\R^d}\phi(t,x) \rd \rho_{t_0}(x)\right) \rd t.
$$
Since we consider $\psi=\psi(x)$ only and the limit $\mu$ actually does not act in time we get
\begin{multline*}
A_n\to \left(\int_0^1\left(\int_{\R^d}(|\nabla \psi|^2 +|\psi|^2) \rd \rho_{t_0}\right) \rd t\right)^{\frac{1}{2}}\\
=\left(\int_{\R^d}(|\nabla \psi|^2 +|\psi|^2) \rd \rho_{t_0}\right)^{1/2}=\|\psi\|_{H^1(\rd \rho_{t_0})}.
\end{multline*}
Thus for fixed $\psi\in \mathcal{C}^\infty_c(\R^d)$ and at any point of differentiability $t_0\in [0,1]$ we get that
$$
\left|\frac{d\Psi}{dt}(t_0)\right|\leq L\|\psi\|_{H^1(\rd \rho_{t_0})}.
$$
The distribution $\zeta(t_0)\in \mathcal{D}'(\R^d)$ defined by $\langle\zeta(t_0),\psi\rangle_{\mathcal D',\mathcal D}=\frac{d\Psi}{d t}(t_0)$
is thus a continuous linear form on $H^1(\rd \rho_{t_0})$ with $\|\zeta(t_0)\|_{H^{-1}(\rd \rho_{t_0})}\leq L$.
By the Riesz representation theorem in $H^1(\rd \rho_{t_0})$ we can then find a unique potential $\u_{t_0}\in H^1(\rd \rho_{t_0})$ such that
\begin{equation} \label{e:risz}
\zeta(t_0)=\left(\u_{t_0},\cdot\right)_{H^1(\rd \rho_{t_0})}\quad \Rightarrow \quad -\dive(\rho_{t_0}\nabla u_{t_0})+\rho_{t_0}u_{t_0}=\zeta(t_0)\quad\mbox{in }\mathcal{D}'(\R^d),
\end{equation}
and 
$$
\|\u_{t_0}\|_{H^1(\rd \rho_{t_0})}= \|\zeta(t_0)\|_{H^{-1}(\rd \rho_{t_0})}\leq L
$$
(if $\rho_{t_0}=0$ in $\Mm(\R^d)$ we simply define $\u_{t_0}\equiv 0$). 

Recalling that $\Psi(t)=\int_{\R^d}\psi \rd \rho_t$ is a.e. differentiable, we see by \eqref{e:risz} that
\begin{equation}
\frac{d}{dt}\int_{\R^d}\psi \rd \rho_t=\int_{\R^d}(\nabla\psi \cdot \nabla u_t + \psi u_t) \rd \rho_t\qquad\mbox{a.e. }t\in (0,1).
\label{eq:diff_lipschitz_ae}
\end{equation}
A subtle issue is that the ``almost every $t\in (0,1)$'' set of differentiability of $\Psi$ might depend here on the choice of $\psi$. Arguing by density and separability as in \cite[Thm. 13.8]{villani08oldnew} we can conclude nonetheless that \eqref{eq:diff_lipschitz_ae} holds for any $\psi\in \mathcal{C}^\infty_c(\R^d)$, which is of course an admissible weak formulation for $\partial_t \rho_t+\dive(\rho_t\nabla u_t)=\rho_t u_t$.
Moreover by construction we have 
$$
\|\u_t\|_{H^1(\rd\rho_t)}\leq L \qquad \mbox{ a.e. }t\in (0,1)
$$
as desired.

Conversely, assume that $(\rho_t,\u_t)_{t\in [0,1]}$ solves the non-conservative continuity equation with $\|\u_t\|_{H^1(\rd\rho_t)}\leq L$ for a.e. $t\in (0,1)$ and that $\rho\in \mathcal{C}_w([0,1];\Mm)$.
Then clearly the total energy $E\leq L^2$.
For any $0\leq t_0\leq t_1$ it is easy to scale in time and connect $\rho_{t_0},\rho_{t_1}$ with cost $(t_1-t_0)\int_{t_0}^{t_1}\|\u_t\|^2_{H^1(\rd\rho_t)}\rd t$, thus $d^2(\rho_{t_0},\rho_{t_1})\leq L^2|t_1-t_0|^2$ and the proof is complete.
\end{proof}

\subsection{Lower semicontinuity of the metric with respect to the weak-$*$ topology}

\begin{defi} Let $(X,\varrho)$ be a metric space, $\sigma$ be a locally compact Hausdorff topology on $X$. We say that the distance $\varrho$ is sequentially lower semicontinuous with respect to $\sigma$ if
for all $\sigma$-converging sequences
$x_k\overset{\sigma}{\rightarrow}x$, $y_k\overset{\sigma}{\rightarrow}y$
one has \begin{equation*}\label{e:lsq1}\varrho(x,y)\leq \liminf_{k\to \infty} \varrho(x_k,y_k).\end{equation*} \end{defi}

\begin{theo} \label{theo:d_LSC_weak*} The distance $d$ is sequentially lower semicontinuous with respect to the weak-$*$ topology on $\Mm$.
\end{theo}

\begin{proof}
Consider any two converging sequences
$$
\rho_0^k\underset{k\to\infty}{\rightarrow}\rho_0,\qquad  \rho_1^k\underset{k\to\infty}{\rightarrow}\rho_1\qquad \mbox{weakly-}*
$$
of finite Radon measures from $\Mm(\R^d)$. For each $k$, the endpoints $\rho_0^k$ and $ \rho_1^k$ can be joined by an admissible narrowly continuous path $(\rho^k_t,\u^k_t)_{t\in [0,1]}$ with energy
$$
E[\rho^k;\u^k]\leq d^2(\rho_0^k,\rho_1^k)+k^{-1}.
$$
Due to weak-$*$ compactness, the masses $m_0^k=d\rho^k_0(\R^d)$ and $m_1^k=d\rho^k_1(\R^d)$ are bounded uniformly in $k\in \N$. By Corollary \ref{corm1} the set $\cup_{k\in\N}\{\rho_0^k,\rho_1^k\}$ is bounded in $(\Mm,d)$, thus the energies $E[\rho^k;\u^k]$ and the masses $m_t^k=\rd\rho^k_t(\R^d)$ are bounded uniformly in $k\in \N$ and $t\in [0,1]$
$$
m_t^k\leq M
\quad\mbox{and}\quad
E[\rho^k;\u^k]\leq \o E.
$$
By the (classical) Banach-Alaoglu theorem with $\Mm\subset(\mathcal{C}_0)^*$, all the curves $(\rho_t^k)_{t\in [0,1]}$ lie in a fixed weak-$*$ sequentially relatively compact set $\mathcal{K}_M=\{\rho\in \Mm:\,\rd\rho(\R^d)\leq M\}$ uniformly in $k,t$. 
By the fundamental estimate \eqref{eq:fundamental_dBL_estimate} we get
$$
\left|\int_{\R^d}\phi(\rd\rho^k_t-\rd\rho^k_s)\right|\leq \sqrt{M \o E}|t-s|^{1/2}(\|\nabla\phi\|_{\infty}+\|\phi\|_{\infty})
$$
for all $\phi\in \mathcal{C}^1_b$, which implies
$$
\forall\,t,s\in [0,1],\,\forall k\in \N:\qquad d_{BL}(\rho^k_s,\rho^k_t)\leq C|t-s|^{1/2}.
$$
Invoking the completeness of $(\Mm(\R^d),d_{BL})$, the above uniform $1/2$-H\"{o}lder continuity w.r.t. $d_{BL}$, the sequential lower semicontinuity of $d_{BL}$ with respect to the weak-$*$ convergence (Lemma~\ref{lem:d_BL_lsc_w*} in the Appendix), and the fact that $\rho^k_t\in \mathcal{K}_M$, we conclude by a refined version of Arzel\`a-Ascoli theorem \cite[Prop. 3.3.1]{AGS06} that there exists a $d_{BL}$ (thus narrowly) continuous curve $(\rho_t)_{t\in [0,1]}$ connecting $\rho_0$ and $\rho_1$ such that
\begin{equation}
\label{eq:pointwise_CV_w*}
\forall t\in [0,1]:\qquad \rho^k_t\to \rho_t\quad\mbox{ weakly-}*
\end{equation}
along some subsequence $k\to\infty$ (not relabeled here).
Let $Q:=(0,1)\times\R^d$ and $\mu^k$ be the measure on $Q$ defined by duality as
$$
\forall \,\phi\in \mathcal{C}_c(Q):\qquad \int_Q\phi(t,x)\rd\mu^k(t,x)=\int_0^1\left(\int_{\R^d}\phi(t,.)\rd\rho^k_t\right)\rd t.
$$
Exploiting the pointwise convergence \eqref{eq:pointwise_CV_w*} and the uniform bound on the masses $m^k_t\leq M$, a simple application of Lebesgue's dominated convergence guarantees that
$$
\mu^k\to \mu^0\qquad \mbox{ weakly-}*\mbox{ in }{\Mm}(Q),
$$
where the finite measure $\mu^0\in \Mm(Q)$ is defined by duality in terms of the weak-$*$ limit $\rho_t=\lim \rho^k_t$ (as was $\mu^k$ in terms of $\rho^k_t$).
Then the energy bound reads
$$
\| \u^k\|^2_{L^2(0,1;H^1(\rd\rho^k_t)}=E[\rho^k;\u^k]\leq d^2(\rho_0^k,\rho_1^k)+k^{-1}\leq C.
$$

We are going to apply a variant of the Banach-Alaoglu theorem, Proposition \ref{Ban} in the appendix, in the space
$$
X=\mathcal{C}^1_c(Q).
$$
Namely, we set
$$
\|\phi\|=\|\nabla\phi\|_{L^\infty(Q)}+\|\phi\|_{L^\infty(Q)},
$$
$$
\|\phi\|_k=\left(\int_{Q}\left(|\nabla\phi|^2+ |\phi|^2 \right)\,\rd\mu^k\right)^{1/2},\qquad k=0,1,\dots,
$$
and define the linear forms
$$
\varphi_k(\phi)=\int_{Q}(\nabla u^k\cdot \nabla\phi+u^k \phi) \,\rd\mu^k,\qquad k=1,2,\dots.
 $$ 
The separability of $\mathcal{C}^1_c(Q)$, the weak-$*$ convergence of $\mu^k$, uniform boundedness of the masses of $\mu^k(Q)\leq M$, and the Cauchy-Schwarz inequality imply that the hypotheses of our Proposition \ref{Ban} are met with
$$
c_k:=\|\varphi_k\|_{(X,\|.\|_k)^*}\leq \|\u^k\|_{L^2(0,1;H^1(\rd\rho^k))}= \sqrt{E[\rho^k;\u^k]}\leq \sqrt{ d^2(\rho_0^k,\rho_1^k)+k^{-1}}.
$$ 
Consequently, there exists a continuous functional $\varphi_0$ on the space $(X,\|\cdot\|_0)$
 such that up to a subsequence 
$$
\forall \phi\in\mathcal{C}^1_c(Q):\qquad
\int_0^1\left(\int_{\R^d}\left\{\nabla u_t^k\cdot \nabla\phi(t,.)+ u_t^k\phi(t,.)\right\}\,\rd\rho^k_t\right) \rd t \underset{k\to\infty}{\rightarrow} \varphi_0(\phi)
$$
with moreover
\begin{equation}\label{e:phin}
\|\varphi_0\|_{(X,\|\cdot\|_0)^*}\leq \liminf_{k \to \infty} d(\rho_0^k,\rho_1^k).
\end{equation}

Let $N_0\subset X$ be the kernel of the seminorm $\|\cdot\|_0$.
By the Riesz representation theorem, the dual $(X,\|\cdot\|_0)^*=(X/N_0,\|\cdot\|_0)^*$ can be isometrically identified with the completion $\overline {X/N_0}$ of ${X/N_0}$ with respect to  $\|\cdot\|_0$, which is exactly $L^2(0,1;H^1(\rd\rho_t))$.
As a consequence there exists $\u=(u,\nabla u)\in L^2(0,T;H^1(\rd\rho_t))$ such that
$$
\varphi_0(\phi)=\int_{Q}\left(\nabla u\cdot \nabla\phi+ u \phi\right) \,\rd\mu^0=\int_0^1\left(\int_{\R^d}\left(\nabla u\cdot \nabla\phi+ u \phi\right)\rd\rho_t\right)\rd t
 $$
 and
 $$
 \|\u\|_{L^2(0,1;H^1(\rd\rho_t))}=\|\varphi_0\|_{(X,\|\cdot\|_0)^*},
 $$
 and it is straightforward to check that $(\rho,\u)$ is an admissible curve joining $\rho^0,\rho^1$ (the above convergence is enough to take the limit in the weak formulation of the non-conservative continuity equation).
Recalling \eqref{e:phin}, it remains to take into account that by the definition of our distance
\begin{align*}
d^2(\rho_0,\rho_1)\leq E[\rho;\u]=\|\u\|^2_{L^2(0,1;H^1(\rd\rho_t))}=\|\varphi_0\|_{(X,\|\cdot\|_0)^*}^2\leq \liminf\limits_{k\to\infty}d^2(\rho^k_0,\rho^k_1).
\end{align*}
\end{proof}
%
\subsection{Existence of geodesics}
We are now in position of proving an important result, namely the existence and characterization of geodesics for our metric structure. We give two proofs: one is more constructive and inspired by the optimal transport theory, and another one is shorter and more abstract. 

\begin{theo}
\label{theo:exist_geodesics}
$(\Mm,d)$ is a geodesic space, and for all $\rho_0,\rho_1\in \Mm$ the infimum in \eqref{e:mini} is always a minimum. Moreover this minimum is attained for a (narrowly continuous) curve $\rho$ such that $d(\rho_t,\rho_s)=|t-s|d(\rho_0,\rho_1)$ and a potential $\u\in L^2(0,1;H^1(\rd\rho_t))$ such that $\|\u_t\|_{H^1(\rd\rho_t)}=cst=d(\rho_0,\rho_1)$ for a.e. $t\in [0,1]$.
\end{theo}
In Section \ref{section:Riemannian_structure} we will show that $(\Mm,d)$ can be viewed as a formal Riemannian manifold, with tangent plane $T_{\rho}\Mm=\{\frac{d \rho}{dt}=-\dive(\rho\nabla u)+\rho u:\,u\in H^1(\rd \rho)\}$ and $\left\|\frac{d\rho}{dt}\right\|_{T_{\rho}\Mm}=\|\u\|_{H^1(\rd\rho)}$.
In this perspective Theorem~\ref{theo:exist_geodesics} can be interpreted as $\left\|\frac{d\rho_t}{dt}\right\|_{T_{\rho_t}\Mm}=cst=d(\rho_0,\rho_1)$, which should be expected along constant-speed geodesics $(\rho_t)_{t\in [0,1]}$ connecting $\rho_0,\rho_1$.

\begin{proof}[Proof 1 via time reparametrization]
Fix any $\rho_0,\rho_1$ and let $\left(\rho^k_t,\u^k_t\right)_{t\in [0,1]}$ be a minimizing sequence in $E[\rho^k;\u^k]=\int_0^1\|\u^k_t\|^2_{H^1(\rd\rho^k_t)}\rd t\underset{k\to\infty}{\rightarrow}d^2(\rho_0,\rho_1)$.
We first claim that we can assume without loss of generality
\begin{equation}
\label{eq:d(rho^k_t0,rho^k_t1)}
\forall \,0\leq t_0\leq t_1\leq 1:\qquad
d^2(\rho^k_{t_0},\rho^k_{t_1})\leq |t_0-t_1|^2 E[\rho^k;\u^k].
\end{equation}
Indeed using a simple arclength reparametrization (Lemma~\ref{lem:Lipschitz_time_arclength} in the Appendix) we can assume that $\|\u^k_t\|^2_{H^1(\rd\rho^k_t)}=cst=E[\rho^k;\u^k]$ is constant in time for all $k\in \N$.
Scaling in time $t=t_0+(t_1-t_0)s$ as before, we get by definition of our distance and Remark~\ref{rmk:time_scaling} that $d^2(\rho^k_{t_0},\rho^k_{t_1})\leq|t_1-t_0|\int_{t_0}^{t_1}\|\u^k_t\|^2_{H^1(\rd \rho^k_t)} \rd t=(t_1-t_0)^2E[\rho^k;\u^k]$.\\
Now because the energies $E[\rho^k;\u^k]$ are bounded \eqref{eq:d(rho^k_t0,rho^k_t1)} shows that the sequence of curves $\{t\mapsto \rho^k_t\}$ are uniformly $1/2$-H\"older continuous with respect to our metric $d$, and arguing as before the masses $m^k_t\leq M$ are bounded uniformly in $t\in [0,1]$ and $k\in \N$, thus $\rho_t^k$ lie in a fixed weakly-$*$ relatively sequentially compact set $\mathcal{K}_M=\{\rho:\,\rd \rho(\R^d)\leq M\}$.
By Theorem~\ref{theo:d_LSC_weak*} we know that $d$ is lower semi-continuous with respect to weak-$*$ convergence, and applying again the refined Arzel\`a-Ascoli theorem \cite[Prop. 3.3.1]{AGS06} we conclude that there exists a $d$-continuous (thus narrowly continuous) curve $\rho$ such that
$$
\forall \,t\in [0,1]:\qquad \rho^k_t\to\rho_t
\quad\mbox{weakly-}*.
$$
From \eqref{eq:d(rho^k_t0,rho^k_t1)} this also implies
$$
d^2(\rho_{t_0},\rho_{t_1})\leq \liminf\limits_{k\to\infty}|t_0-t_1|^2 E[\rho^k;\u^k]=|t_1-t_0|^2d^2(\rho_0,\rho_1)
$$
for all $t_0,t_1\in [0,1]$.
By the triangular inequality we conclude that in fact
$$
\forall\, t_0,t_1\in [0,1]:\qquad \qquad d(\rho_{t_0},\rho_{t_1})=|t_1-t_0|d(\rho_0,\rho_1),
$$
and in particular the curve $t\mapsto\rho_t$ is $L$-Lipschitz with $L=d(\rho_0,\rho_1)$. 

Applying Theorem~\ref{theo:lipschitz_curves} we see that there exists a potential $\u_t$ solving the non-conservative continuity equation such that $\|\u_t\|_{H^1(\rd\rho_t)}\leq L$ a.e. $t\in (0,1)$, and because
$$
L^2 =d^2(\rho_0,\rho_1)\leq E[\rho;\u]=\int_0^1\|\u_t\|_{H^1(\rd\rho_t)}^2\mathrm{d}t\leq L^2
$$
we conclude that in fact $\|\u_t\|_{H^1(\rd\rho_t)}=L=d(\rho_0,\rho_1)$ a.e. $t\in (0,1)$.
\end{proof}

\begin{proof}[Proof 2 via midpoints] 
We first observe from the definition of our distance that $(\Mm,d)$ is a length space (e.g. by an application of the almost midpoint criterion, \cite[Thm. 2.4.16(2)]{Bur}.). By Corollary \ref{corm} and the (classical) Banach-Alaoglu theorem,  the weak-$*$ topology is $d$-boundedly compact. Now Lemma \ref{hrt} (variant of the Hopf-Rinow theorem in the Appendix) together with Theorem \ref{theo:equiv} and Theorem~\ref{theo:d_LSC_weak*} imply that $(\Mm,d)$ is a geodesic space. The existence and claimed properties of the minimizer follow from Theorem~\ref{theo:lipschitz_curves} exactly as in Proof 1.
\end{proof}
\begin{rmk}
The geodesics can be non-unique, see Section \ref{ss3}. 
\end{rmk}
%
%
%
\section{Underlying geometry}
\label{section:Riemannian_structure}
We show here that our distance $d$ endows $\Mm(\R^d)$ with a formal Riemannian structure. We closely follow Otto's approach \cite{otto01}, which was originally developed for the optimal transport of probability measures in the Wasserstein metric space $(\mathcal{P}_2,\mathcal{W}_2)$. Refer the reader to the particularly clear exposition in \cite{villani03topics,villani08oldnew}.
%
\subsection{Otto's Riemannian formalism}
Let us recall that if $\mathcal{M}$ is a smooth differential manifold then the tangent plane $T_x\mathcal{M}$ at a point $x\in M$ can be viewed as the vector space of tangent vectors $\left.\frac{dx_t}{dt}\right|_{t=0}$ of all $\mathcal{C}^1$ curves $t\mapsto x_t\in \mathcal{M}$ passing through $x(0)=x$.
By Theorem~\ref{theo:lipschitz_curves} and Theorem~\ref{theo:exist_geodesics} we know that, given any fixed endpoint $\rho\in \Mm$, there always exists a constant-speed geodesic $\rho_t$ connecting $\rho|_{t=0}=\rho$ to arbitrary $\nu\in \Mm$ and parametrized as
$$
\left\{
\begin{array}{ll}
\partial_t\rho_t=-\dive(\rho_t \nabla u_t)+\rho_t u_t & \mbox{in }\mathcal{D'}((0,1)\times\R^d)\\
\|\u_t\|_{H^1(\rd\rho_t)}=cst=d(\rho,\nu) & \mbox{a.e. }t\in (0,1).
\end{array}
\right.
$$
Since $t\mapsto\rho_t$ is a constant-speed geodesic in $(\Mm,d)$ it should be a $\mathcal{C}^1$ curve, and, with a slight abuse of notation, we naturally identify the $\partial_t\rho_t$ distributional term above with a tangent vector $\zeta_t=\frac{d\rho_t}{dt}\in T_{\rho_t}\Mm$.
According to $\zeta_t=-\dive(\rho_t\nabla u_t)+\rho u_t$ we see that the tangent vector $\zeta_t=\frac{d\rho_t}{dt}$ corresponds to a potential $\u_t=(u_t,\nabla u_t)\in H^1(\rd\rho_t)$.
Assuming for simplicity that $\frac{d\rho_t}{dt},\u_t$ can be somehow evaluated at $t=0^+$, we see that any geodesic passing through the fixed endpoint $\rho|_{t=0}=\rho$ gives a tangent vector $\left.\frac{d\rho_t}{dt}\right|_{t=0}$, which thus corresponds to some $\u|_{t=0}\in H^1(\rd\rho)$. This strongly suggests to define the tangent space in terms of potentials as
$$
T_\rho\Mm:=\{\zeta=-\dive(\rho\nabla u)+\rho u:\quad \u=(u,\nabla u)\in H^1(\rd\rho)\}
$$
and
$$
\|\zeta\|_{T_\rho\Mm}:=\|\u\|_{H^1(\rd\rho)}=\left(\int_{\R^d}(|\nabla u|^2+|u|^2)\rd\rho\right)^{1/2}.
$$
Observe that, given $\zeta$ and ignoring all smoothness issues, the elliptic PDE
\begin{equation} \label{e:epde}
-\dive(\rho\nabla u)+\rho u=\zeta
\end{equation}
is coercive in $H^1(\rd\rho)$ so the correspondence between tangent vectors $\zeta$ and potentials $\u$ is at least formally uniquely defined.
By polarization this automatically defines the Riemannian metric on $T_\rho\Mm$
$$
\left<\frac{d \rho}{d t_1},\frac{d \rho}{d t_2}\right>_{\rho}:=\left<\u_1,\u_2\right>_{H^1(\rd\rho)}=\int_{\R^d}(\nabla u_1\cdot \nabla u_2+u_1u_2)\rd\rho,
$$
where the subscript $\rho$ in the left-hand side highlights the dependence on the base point $\rho\in \Mm$ in the weighted $H^1(\rd\rho)$ scalar product. A formal but useful remark is that 
\begin{multline} \label{usef}
\left<\frac{d \rho}{d t_1},\frac{d \rho}{d t_2}\right>_{\rho}=\int_{\R^d}(\nabla u_1\cdot \nabla u_2+u_1u_2)\rd\rho\\
=\int_{\R^d}u_2\zeta_1\rd x =\int_{\R^d} u_2\partial _{t_1}\rho \,\rd x
=\int_{\R^d}u_1\zeta_2\rd x=\int_{\R^d} u_1\partial _{t_2}\rho\,\rd x .
\end{multline}

Observe that with this formalism our definition \eqref{e:mini} simply reads
$$
d^2(\rho_0,\rho_1)=\inf\left\{\int_0^1\left\|\frac{d\rho_t}{dt}\right\|_{T_{\rho_t}\Mm}^2\,\mathrm{d}t\right\},
$$
where the infimum is taken over all admissible curves connecting $\rho_0,\rho_1$.
This is nothing but the classical Lagrangian formulation of (squared) distances in a reasonably smooth Riemannian manifold.

\begin{rmk}
It is worth pointing out that $\rho=0$ is a degenerate point in our pseudo-manifold $\Mm$ since the tangent space $T_0\Mm$ is zero-dimensional.
\end{rmk}

%
\subsection{Differential calculus in $(\Mm,d)$ and induced dynamics}
Now that we defined a Riemannian structure on $\Mm$ in terms of the distance $d$, one would like to differentiate functionals $\mathcal F:\Mm\to \R$, or in other words compute gradients ''$\grad_d\mathcal{F}(\rho)$`` with respect to this Riemannian structure.
This has already been done in the setting of optimal transport in the Wasserstein space $(\mathcal{P}_2,\mathcal{W}_2)$, see \cite{otto01,villani03topics} and references therein.
The approach is once again very similar, and still formal.

Let us recall that if $(\mathcal{M},g)$ is a smooth Riemannian manifold and $\mathcal{F}:\mathcal{M}\to\R$, the Riemannian metric tensor $g$ and the Riesz representation theorem allow to pull back the differential $D\mathcal{F}(x)\in T^*_x\mathcal{M}$, which is a cotangent vector, to a unique tangent vector $\grad \mathcal{F}(x)\in T_x\mathcal{M}$:
For all $\mathcal{C}^1$ curves $t\mapsto x_t$ passing through $x|_{t=0}=x$ with tangent vector $\left.\frac{dx_t}{dt}\right|_{t=0}=\zeta$ the gradient $\grad\mathcal{F}$ is defined by duality via the chain rule
$$
\left.\frac{d}{dt}\mathcal{F}(x_t)\right|_{t=0}=\left<D\mathcal{F}(x),\zeta \right>_{T^*_x\mathcal{M},T_x\mathcal{M}}=g_x(\grad \mathcal{F}(x),\zeta).
$$

Mimicking this computation and washing out all smoothness issues, it is easy to deduce here that the gradients of functionals $\mathcal{F}:\rho\in \Mm\to\R$ are at least formally given by
\begin{equation}
\label{eq:formula_gradient_d}
\begin{array}{c}
 \grad_d\mathcal{F}(\rho)=-\dive\left(\rho\nabla\frac{\delta \mathcal{F}}{\delta\rho}\right)+\rho\frac{\delta \mathcal{F}}{\delta\rho}\\
 \left\|\grad_d\mathcal{F}(\rho)\right\|_{T_\rho\Mm}=\left\|\frac{\delta\mathcal{F}}{\delta\rho}\right\|_{H^1(\rd\rho)}
 \end{array},
\end{equation}
 where $\frac{\delta \mathcal{F}}{\delta\rho}$ denotes the first variation with respect to the usual $L^2$ Euclidean structure (e.g., if $\mathcal{F}(\rho)=\int U(x)d\rho(x)$, then $\frac{\delta \mathcal{F}}{\delta\rho}=U$).\\

With gradients available one can obtain induced dynamics on $(\Mm,d)$, the main examples being gradient flows and Hamiltonian flows. More precisely, given some functional $\mathcal{F}(\rho)$ on $\Mm$ the $\mathcal{F}$-gradient flow is defined as
$$
\frac{d\rho}{dt}=-\grad_d\mathcal{F}(\rho).
$$
Similarly, Hamiltonian systems read
$$
\frac{d\rho}{dt}=-J\cdot\grad_d\mathcal{F}(\rho),
$$
where $J$ is a given Hamiltonian antisymmetric operator (i-e a closed $2$-form on $T_\rho\Mm$). As an application of this Riemannian formalism we will study in Section~\ref{section:animals} a particular gradient flow, and exploit the formal structure to derive new long-time convergence results by means of entropy-entropy dissipation inequalities.\\

Second order differential calculus can also be developed: the Hessian can be formally defined as
\begin{equation}
\label{eq:Hessians_geodesics}
\left<\operatorname{Hess}_d\mathcal{F}(\rho)\cdot\zeta,\zeta\right>_\rho=\left.\frac{d^2}{dt^2}\right|_{\mbox{geod}}\mathcal{F}(\rho_t),
\end{equation}
where differentiation should be performed along a geodesic path $\rho_t$ passing through $\rho$ with tangent vector $\zeta$ at time $t=0$. In order to exploit this formula one needs to compute the geodesics with prescribed initial position and velocity, and the first issue is thus the very existence of these objects. As for the practical computation itself, consider for example $\mathcal{F}(\rho)=\int_{\R^d} \rd\rho$: the second time derivative is
\begin{multline*}
\int_{\R^d} \partial^2_{tt} \rho_t=\int_{\R^d} {\partial}_t(-\nabla\cdot(\rho_t \nabla u_t)+\rho_t u_t)=\int_{\R^d} {\partial}_t(\rho_t u_t)\\
=\int_{\R^d} u_t\partial_t \rho_t + \rho_t \partial_t u_t= \int _{\R^d}\rho_t (|\nabla u_t|^2+u_t^2 +\partial_t u_t).
\end{multline*} 
Thus the second and key issue is to find an equation yielding an expression for $\partial_t u_t$ or $\p_t(\rho_t u_t)$ in terms of $\rho$, $u$ and their spatial derivatives.
The next two subsections are devoted to these questions. We will see that in the definition of the Hessian it is natural to replace the metric geodesics by some objects which we call \emph{trajectory geodesics}.
The latter possess fine minimization properties, exist for given initial value $\rho$ and velocity $\zeta$, and are described by nice PDEs, which, in particular, provide the required expressions of $\partial_t u_t$ and ${\partial}_t(\rho_t u_t)$ in terms of $\rho$ and $u$.
%
%
\subsection{Trajectory geodesics}

\label{tdg}
In order to introduce the concept of trajectory geodesic, we use a heuristic particle interpretation of the considered dynamics, which better illustrates the underlying ideas. 

Think of movement of charged particles in $\R^d$, whose mass is fixed but whose charge can possibly evolve according to some a priori given law described later on. 
Assume first that we have a finite number $i=1\ldots N$ of moving particles with position $x_t(i)$ and charge $k_t(i)$. In terms of densities and curves in $\Mm$ this dynamics can be formalized as
\begin{equation*}
\label{eq:ansatz}
\rho_t=\sum\limits_{i=1}^N k_t(i)r_t(i)\qquad \mbox{with}\qquad r_t(i):=\delta_{x_t(i)}.
\end{equation*}
Given a reasonably smooth potential $u_t(x)$, the function $\rho_t$ is the solution of the non-conservative continuity equation $\p_t\rho_t+\dive(\rho_t\nabla u_t)=\rho_t u_t$ if and only if
$$
\frac{d}{dt}x_t(i)=\nabla u_t(x_t(i))
$$
and 
$$
\frac d {dt }(\log k_t(i))=u_t(x_t(i))
$$
along the trajectory of each particle (this claim can be checked employing e.g. \cite[Prop. 3.6]{maniglia2007probabilistic}, see however Remark \ref{r:npartt} below).
Define the energy of each individual particle by
 \begin{equation}
 \label{e:eip}
 E_i=  \int_0^1 k_t(i) \left[ \left|\frac d {dt} (\log k_t(i))\right|^2 +\left|\frac d {dt}\,x_t(i)\right|^2\right] \,dt.
 \end{equation}
 Then the total energy sums as
\begin{equation}
\label{e:esum}
E=\sum\limits_{i=1}^N E_i=\int_0^1\left(\int_{\R^d}(|\nabla u_t|^2+|u_t|^2)\rd\rho_t\right)\rd t .
\end{equation}

If the pair $(\rho_t,\u_t)_{t\in [0,1]}$ were to be a geodesic (in the sense of Theorem \ref{theo:exist_geodesics}), then it should minimize the total energy $E$ for fixed $\rho_0$ and $\rho_1$, and thus also minimize each of the $E_i$'s: more precisely, there should hold
\begin{multline}
\label{e:min01}
E_i=\min\Bigg\{ \int_0^1 \tilde k_t \left[ \left|\frac d {dt} (\log \tilde k_t)\right|^2 +\left|\frac d {dt}\,\tilde x_t\right|^2\right] \,dt:\\ \tilde x_0(i)=x_0(i),\ \tilde x_1(i)=x_1(i), \quad \tilde k_0(i)=k_0(i), \tilde k_1(i)=k_1(i)\Bigg\}.
\end{multline}
It turns out to be relevant to replace the condition of being a geodesic by that of being a minimizer of each $E_i$ as in \eqref{e:min01}.
\begin{rmk} \label{r:npartt}
The above discussion is legitimate, e.g., if we have strictly $N$ different particles for all times, i-e when the trajectories of the particles do not overlap and no mass splitting or coalescence occur.
Much more general scenarios allowing for trajectory crossing are certainly possible, provided that the potential $u$ loses spatial smoothness and the non-conservative continuity equation is understood in the weak sense.
This is of course a delicate issue, which we shall ignore below by staying at the formal level.
\end{rmk}

We carry out now a similar formal analysis in a more general framework, considering general measures $\rho_t$ as the superposition of possibly infinitely many indivisible infinitesimal particles.

Let $(\rho_t,\u_t)_{t\in [0,1]}$ be a generic admissible path in the sense of Definition \ref{d:metr}, and assume that the velocity field $\nabla u_t$ is smooth enough.
Let $T_t$ be $t$-flow generated by the characteristic ODE $\frac{dy}{dt}=\nabla u_t(y)$, and define $r_t:\R^d\to \R$ as the pushforward 
\begin{equation}
\label{rpart}
r_t=T_t\#\rho_0.
\end{equation}
Then $r_t$ is a curve in $\Mm$ with constant mass $m_t=\rd\rho_0(\R^d)$, and
\begin{equation}
\label{geod2}
\p_t r_t+\dive (r_t\nabla u_t)=0.
\end{equation}  

As in the previous discrete setting, we think of $\rho_t$ as the charge density of moving charged particles and decompose it as a product
$$
\rho_t(x)=r_t(x)k_t(x).
$$
Here $r_t(x)$ defined by \eqref{rpart} is the particle density at time $t$ (density of particles without thinking of their charge), while $k_t(x)$ is the charge initially normalized as $k_0(x)\equiv 1$. When the particles move along their trajectories, they continuously change their charge according to
\begin{equation}
\label{geod1} 
\p_t\rho_t+\dive (\rho_t \nabla u_t)=\rho_t u_t.
\end{equation} 

Denote the Lagrangian time derivative by
$$
\frac{D}{Dt}=\p_t{} + \nabla u \cdot\nabla.
$$
Multiplying \eqref{geod2} by $k$ and subtracting the result from \eqref{geod1} it is easy to obtain $r_t\frac {Dk_t} {Dt} =r_tk_t u_t$, thus formally 
\begin{equation}
\label{geod3}
\frac D {Dt}(\log k_t)=u_t
\end{equation} 
as in the previous discrete setting of finitely many particles.

Now switch to the Lagrangian framework and let $x_t(X):=T_t(X)$ be the trajectory of a particle starting at $x_{0}(X)=X$. From \eqref{geod3} we see that
\begin{equation}
\label{geod5}
\frac D {Dt}\, x_t=\nabla u_t(x_t)
\quad \mbox{and}\quad
\frac D {Dt} (\log k_t(x_t))=u_t(x_t).
\end{equation} 
Thus, the infinitesimal energy along the trajectory of a single particle $x_t(X)$ is proportional to   

\begin{equation*}
E_X=\int_0^1 k_t \left[ \left|\frac D {Dt} (\log k_t)\right|^2 +\left|\frac D {Dt}\,x_t\right|^2\right] \,\rd t,
\end{equation*}
which is of course the continuous Lagrangian counterpart of \eqref{e:eip}.
For fixed $X$ the functions $x(t)=x_t(X)$ and $k(t)=k_t(x_t(X))$ depend only on time, $\frac D {Dt}$ becomes the ordinary time derivative, and the energy reads
\begin{equation}
\label{gg1}
E_X(x,k)=  \int_0^1\left( \frac {|k'|^2} {k}+k|x'|^2 \right)\,dt.
\end{equation}
The boundary conditions for $x,k$ are  naturally 
\begin{equation}
\label{bk1}
x_0=X,\ x_1=T_1(X),\qquad \ k_0=1,\ k_1=\frac{\rho_1(T_1(X))}{r_1(T_1(X))}.
\end{equation}

We temporarily forget about the origin of the functions $x,k$ and consider the minimization problem
\begin{equation}
\label{e:min1}
E_X(x,k)=\min E_X(\tilde x,\tilde k):\ \tilde x,\tilde k \mathrm{\, satisfy\,  \eqref{bk1}},\ \tilde k_t>0\ \mathrm{for}\ 0<t<1.
\end{equation}

\begin{defi} 
An admissible curve $(\rho_t,\u_t)_{t\in [0,1]}$ is called a \emph{trajectory geodesic} if the corresponding $x,k$ solve the minimization problem \eqref{e:min1} for every $X\in \R^d$.
\end{defi}

Given a trajectory geodesic and $X$, we can write the set of Euler-Lagrange equations for \eqref{e:min1}:
\begin{equation}
\label{el1}
\frac{2k''}{k}-\frac {|k'|^2}{k^2}=|x'|^2,
\end{equation} 
\begin{equation}
\label{el2}
(kx')'=0.
\end{equation}
From \eqref{el2} we see that  $kx'$ is conserved along the trajectories, that is, the trajectories are straight lines in the base space $\R^d$
\begin{equation}
\label{el4}
\left(\frac{x'}{|x'|}\right)'=0
\end{equation}
but the speed $|x'|$ is non-constant and is equal to $cst/k$ along the trajectories. Furthermore, \eqref{el1} can be rewritten in the form 
\begin{equation*}
\label{el3}
2(\log k)''+|(\log k)'|^2=|x'|^2.
\end{equation*}
Recalling \eqref{geod5} we find that  
$2\frac D {Dt}\, u_t+|u_t|^2=|\nabla u_t|^2$,
or equivalently
\begin{equation}
\label{bge1}
\p_t u_t+\frac 1 2 (|u_t|^2+|\nabla u_t|^2)=0.
\end{equation}
Moreover, \eqref{geod5} and \eqref{el4} yield
$\frac D {Dt}\left(\frac{\nabla u_t}{|\nabla u_t|}\right)=0$,
i-e
\begin{equation}
\label{bge2}
\p_t\left(\frac{\nabla u_t}{|\nabla u_t|}\right)+ \nabla u_t \cdot\nabla \left(\frac{\nabla u_t}{|\nabla u_t|}\right)=0.
\end{equation}

In order to calculate Hessians as in \eqref{eq:Hessians_geodesics}, we need a procedure to construct a (trajectory) geodesic starting from any measure $\rho$ with prescribed initial velocity $\zeta\in T_\rho\Mm$.
Let $\u=(u,\nabla u)$ be the corresponding initial potential (formally) determined by $\zeta=-\dive(\rho\nabla u)+\rho u$, and consider the solution $(\rho_t,\u_t)_{t\in [0,\delta]}$ to the following problem: 

\begin{equation} \label{e:geodeq}
\left\{ 
\begin{array}{l} 
\p_t\rho_t=-\dive(\rho_t\nabla u_t)+\rho_tu_t, \\ \p_t u_t=-\frac 1 2 (|u_t|^2+|\nabla u_t|^2),\\
\rho_0=\rho,\ u_0=u.
\end{array}
\right.
\end{equation}

A crucial and somewhat surprising observation is that $u_t$ satisfies \eqref{bge2}: the reason is that \eqref{bge1} always implies \eqref{bge2}. Indeed, we first compute 
 \begin{equation*}
 \frac D {Dt}(\nabla u_t)=-\frac 1 2 \nabla (|u_t|^2+|\nabla u_t|^2)+\nabla^2 u_t\cdot \nabla u_t=-u_t\nabla u_t,\end{equation*}
thus
\begin{multline*}
\frac D {Dt}\left(\frac{\nabla u_t}{|\nabla u_t|}\right)=\frac {\frac D {Dt}(\nabla u_t)}{|\nabla u_t|}-\nabla u_t \frac {(\frac D {Dt} \nabla u_t)\cdot \nabla u_t}{|\nabla u_t|^3}\\
=-\frac {u_t\nabla u_t}{|\nabla u_t|}+\nabla u_t \frac {u_t\nabla u_t\cdot \nabla u_t}{|\nabla u_t|^3} =0.
\end{multline*}

 Repeating and rearranging the argument above, we deduce that for each $X$ the corresponding pair $x,k$ solves the Euler-Lagrange equations \eqref{el1},\eqref{el2}. Observe from \eqref{gg1} that the infinitesimal energy $E_X(x,k)$ is convex in $x,k,k'$: thus at least formally any solution of the Euler-Lagrange equations must be the unique minimizer, which by definition means that $(\rho_t,\u_t)$ is a trajectory geodesic. 

We have just observed that an admissible curve $(\rho_t,\u_t)$ is a trajectory geodesic if and only if it satisfies both the Hamilton-Jacobi equation \eqref{bge1} and the non-conservative continuity equation.
It sounds plausible that a trajectory geodesic should (at least formally) always be a geodesic (since it simultaneously minimizes the energy along every trajectory), but this claim is not necessarily true, see Section \ref{ss3}. The underlying obstacles are subtle and fundamental, as similar factors undermine the Riemannian structure of the Wasserstein space $(\mathcal{P}_2,\mathcal{W}_2)$, see Remark \ref{tgd1}. However, a weaker statement holds true: the trajectory geodesics have constant metric speed in the sense that
\begin{equation*}
\left\|\frac {d \rho_t}{dt}\right\|_{T_{\rho_t}\Mm}=\|\u_t\|_{H^1(\rd\rho_t)}=cst.
\end{equation*}
Indeed, using \eqref{geod1}, \eqref{bge1} and \eqref{usef} we find at least formally for any trajectory geodesic
\begin{multline*}
\frac {d}{dt}\|\u_t\|^2_{H^1(\rd\rho_t)}=2\left\langle \u_t, \frac {d \u_t}{dt} \right\rangle_{H^1(\rd\rho_t)}+\int_{\R^d}(|\nabla u_t|^2+|u_t|^2)\partial_t \rho_t\, \rd x\\
=2\int_{\R^d}\partial_t u_t\partial_t \rho_t\, \rd x+\int_{\R^d}(|\nabla u_t|^2+|u_t|^2)\partial_t \rho_t\, \rd x=0.
\end{multline*}
 
\begin{rmk} \label{tgd1}
A straightforward computation shows that the geodesic equations \eqref{e:geodeq} imply
\begin{equation*} \label{e:geodeqv}
\left\{ 
\begin{array}{l} 
\p_t\rho_t=-\dive(\rho_t\nabla u_t)+\rho_tu_t, \\ \p_t (\rho_t \nabla u_t)=-\dive (\rho_t\nabla u_t \otimes \nabla u_t),\\
\rho_0=\rho,\ u_0=u.
\end{array}
\right.
\end{equation*}
 This problem is very similar to the geodesics equations for optimal time-dependent transport \cite{villani03topics} of probability measures in the Wasserstein framework
\begin{equation}
\label{e:geodeqvw}
\left\{ 
\begin{array}{l} 
\p_t\rho_t=-\dive(\rho_t\nabla u_t), \\ \p_t (\rho_t \nabla u_t)=-\dive (\rho_t\nabla u_t \otimes \nabla u_t),\\
\rho_0=\rho,\ u_0=u,
\end{array}
\right.
\end{equation}
which is a particular case of the sticky particles system \cite{BG98}. 
A trajectory geodesic formalism similar to the above one can of course be developed in the Wasserstein space $(\mathcal{P}_2,\mathcal{W}_2)$ (with $k\equiv 1$, no charge is considered).
The Wasserstein geodesics are trajectory geodesics in that setting. The reason is that the trajectories of the dynamical optimal transport problem have constant velocity \cite{villani03topics}, thus the respective Euler-Lagrange equality $x''=0$ is satisfied. The converse statement is not always true. Note that since the exponential map $\rho_t=\operatorname{exp}_{\rho_0}(t\zeta)$ for the Wasserstein geodesics cannot be properly defined for all velocities, the classical Otto calculus \cite{villani03topics,villani08oldnew,otto01} uses variants of the system \eqref{e:geodeqvw} for the calculation of the Hessian. Hence, from the perspective of our formalism, the Otto calculus implicitly employs trajectory geodesics (without defining them) in place of metric geodesics.  \end{rmk}

\subsection{Hessians and $\lambda$-convexity}
The above discussion suggests the following 
\begin{defi}
The Hessian of a functional $\mathcal{F}:\Mm\to\R$ with respect to our structure is formally defined by
\begin{equation}
\left<\operatorname{Hess}_d\mathcal{F}(\rho)\cdot\zeta,\zeta\right>_\rho=\frac{d^2}{dt^2}\mathcal{F}(\rho_t).
\end{equation}
Here the path $\rho_t$ is determined by \eqref{e:geodeq} where the initial potential $\u=(u,\nabla u)$ is related to the tangent vector $\zeta$ via the elliptic equation $-\dive(\rho\nabla u)+\rho u=\zeta$.
\end{defi}
Then we can give
\begin{defi}
\label{definition:lambda_convexity}
A functional $\mathcal{F}$ on $\Mm$ is convex (resp. $\lambda$-convex, $\lambda\in \R$) with respect to our structure provided
$
\left<\operatorname{Hess}_d\mathcal{F}(\rho)\cdot\zeta,\zeta\right>_\rho\geq 0
$ (resp.,
$
\left<\operatorname{Hess}_d\mathcal{F}(\rho)\cdot\zeta,\zeta\right>_\rho\geq \lambda \left<\zeta,\zeta\right>_\rho=\lambda \|u\|^2_{H^1(\rd\rho)}\text{)}
$ for all $\rho,\zeta$.
\end{defi}

Calculations of the Hessians are rather tedious, and as an example we only present here the final result for the \emph{internal energy} \cite{villani03topics} determined for absolutely continuous measures by
$$
\mathcal{E}(\rho)=\int_{\R^d} E(\rho(x))\mathrm{d}x,
$$
where $E:\R^+\to \R$ is a given measurable energy density.
For this functional one can compute explicitly
\begin{multline*} \left<\operatorname{Hess}_d\mathcal{E}(\rho)\cdot\zeta,\zeta\right>_\rho=\int_{\R^d}\Big\{P(\rho)\Gamma_2(u)+P_2(\rho)|\Delta u|^2-(2P_2(\rho)+P(\rho))u\Delta u\\ +\left(Q_2(\rho)-\frac 12 {Q(\rho)}  -P_2(\rho)\right)|\nabla u|^2+\left(Q_2(\rho)-\frac 12 {Q(\rho)}\right)|u|^2\Big\}.\end{multline*}
Here we have used the notation 
\begin{gather*}
P(\rho)= E'(\rho)\rho-E(\rho),  \qquad   P_2(\rho)= P'(\rho)\rho-P(\rho),\\
Q(\rho)=E'(\rho)\rho, \qquad  Q_2(\rho)=Q'(\rho)\rho,\\ 
\Gamma_2(u)=\sum\limits_{i,j=1}^d|\p^2_{x_ix_j}u|^2.
\end{gather*}
%
%
\subsection{Distance and geodesics between two one-point measures}
\label{ss3}

We conclude this section with a highly illustrative example, which in particular will show that geodesics may be non-unique and that \emph{trajectory} geodesics are not necessarily \emph{metric} geodesics.
Namely, we are going to describe the geodesics and calculate the distance between two one-point measures $\rho_0=k_0\delta_{x_0}$ and $\rho_1=k_1\delta_{x_1}$ with $x_0,x_1\in\R^d$, $k_0,k_1> 0$, and $\xi:=|x_0-x_1|>0$. 
Observe that the case when one of the masses vanishes fits into Proposition \ref{prop:geodesic_rho_to_0}, and the case when $\xi=0$ fits into Proposition \ref{prop:scaling_lambda}. 

In order to send $\rho_0$ to $\rho_1$, a first natural strategy is to treat the dynamics as the movement of a sole indivisible particle $\rho_t=k(t)\delta_{x(t)}$ moving from $x_0$ to $x_1$, as described in Section \ref{tdg}. The corresponding energy consumption for this \emph{transport} strategy is 
\begin{multline}
\label{e:min02}
E_{tr}(k_0,k_1,\xi)=\min\Bigg\{ \int_0^1 \tilde k_t \left[ \left|\frac d {dt} (\log \tilde k_t)\right|^2 +\left|\frac d {dt}\,\tilde x_t\right|^2\right] \,\rd t:\\
\tilde x_0=x_0,\ \tilde x_1=x_1, \quad \tilde k_0=k_0, \tilde k_1=k_1,\quad \tilde k_t>0\ \mathrm{for}\ 0<t<1\Bigg\},
\end{multline}
cf. \eqref{e:min01}. A laborious but rather elementary calculus of the variations shows that this minimization problem has a solution for $$\xi<2\pi,$$ which has the form
\begin{gather}
\label{kt}k_t=a(t-b)^2+c,\\
\label{e:xt}x_t=x_0+2\frac {x_1-x_0}{\xi}\left[\arctan((t-b)\sqrt{a/c})+\arctan(b\sqrt{a/c})\right],\\
\label{e:a0}E_{tr}(k_0,k_1,\xi)=4a,
\end{gather}
where
\begin{gather}
\label{e:a1}
a=k_0+k_1-2\cos(\xi/2)\sqrt{k_0k_1},\\
b=\frac{k_0-\cos(\xi/2)\sqrt{k_0k_1}}{k_0+k_1-2\cos(\xi/2)\sqrt{k_0k_1}},\\
c=\frac{k_0k_1\sin^2(\xi/2)}{k_0+k_1-2\cos(\xi/2)\sqrt{k_0k_1}}\label{e:ccc}.
\end{gather}
Observe that the solution \eqref{e:xt} ceases to exist whenever $\xi=|x_1-x_0|\geq 2\pi$.

Another natural strategy, made possible by the reaction term $\rho_tu_t$ in the non-conservative continuity equation, is to kill all the mass at $x_0$ while simultaneously creating mass at $x_1$. This means that we treat the dynamics as the evolution of two distinct particles $\rho_t=k_t(1)\delta_{x_1}+k_t(2)\delta_{x_2}$ with fixed positions but non-constant charges, with
\begin{gather*}
\label{twop}
k_0(1)=k_0, k_1(1)=0
\quad\mbox{and}\quad
k_0(2)=0, k_1(2)=k_1.
\end{gather*}

The corresponding energy intake for this \emph{stationary} strategy can be calculated using  Proposition \ref{prop:geodesic_rho_to_0}, viz.
\begin{equation}
\label{e:a2}
E_{st}=4(k_0+k_1).
\end{equation}
Note that this strategy also works for $\xi\geq 2\pi$, contrarily to the previous transport one. 

We now make the ansatz that a geodesic joining two one-point measures $\rho_0=k_0\delta_{x_0},\rho_1=k_1\delta_{x_1}$ should use a mixture of these transport and/or stationary strategies, corresponding to an evolution of (at most) three particles: the first one travels between $x_0$ and $x_1$; the second one stays in $x_0$ and its charge goes to zero as time evolves; the third is always at location $x_1$, and builds-up its charge starting form zero at $t=0$ (this geometry may require a loss of smoothness of the driving potential $u$, cf. Remark \ref{r:npartt}). That is,
\begin{gather*}
\rho_t=\sum\limits_{i=1}^3 k_t(i)r_t(i),\qquad
r_t(i)=\delta_{x_t(i)}, \\
x_0(1)=x_0,x_1(1)=x_1,\quad x_0(2)=x_1(2)=x_0,\quad x_0(3)=x_1(3)=x_1,\\
k_0(1)+k_0(2)=k_0,\quad k_1(1)+k_1(3)=k_1,\quad k_1(2)=0, k_0(3)=0.
\end{gather*}
The heuristics behind this ansatz is that indivisible infinitesimal particles which have non-zero charge at the initial and final times are considered to be in the first bulk, the ones which have zero charge at time $t=1$ constitute the second bulk, and the remaining ones go to the third bulk. 
Denoting the independent parameters $k_0(1)$ and $k_1(1)$ by $\gamma_0$ and $\gamma_1$, resp.,  we discover that the total energy intake is
\begin{multline*}
E(\gamma_0,\gamma_1)=E_1+E_2+E_3\\
=4[\gamma_0+\gamma_1-2\cos(\xi/2)\sqrt{\gamma_0\gamma_1}]+4(k_0-\gamma_0)+4(k_1-\gamma_1)\\
=4[k_0+k_1-2\cos(\xi/2)\sqrt{\gamma_0\gamma_1}]
\end{multline*}
(here we implicitly restrict to $\xi<2\pi$ so that both strategies are feasible).
For $\xi<\pi$ one can check that the minimum of $E(\gamma_0,\gamma_1)$ is achieved for $\gamma_0=k_0,\gamma_1=k_1$, i-e the pure transport strategy is optimal. For $\xi>\pi$, the energy minimum is achieved at $\gamma_0=\gamma_1=0$, i-e the geodesic only uses the pure stationary strategy.
At the critical value $\xi=\pi$, any combination of the two strategies (i-e any choice of the free parameters $\gamma_0\in[0,k_0]$ and $\gamma_1\in [0,k_1]$) does an optimal job, and in particular we see that the geodesics are definitely not unique.
Furthermore, for $\pi<\xi<2\pi$, the pure transport curve determined by \eqref{kt}--\eqref{e:ccc} is a trajectory geodesic by construction, but not a metric geodesic (since the stationary strategy is less expensive in terms of energy).

Let us now consider the particular case of probability measures $k_0=k_1=1$ and compare $d^2(\delta_{x_0},\delta_{x_1})$ with the Wasserstein distance $\mathcal{W}_2^2(\delta_{x_0},\delta_{x_1})$, which is clearly equal to $\xi^2=|x_1-x_0|^2$ in this example.
Remember that by Proposition \ref{prop:d<W2} our distance does not  a priori  exceed the Wasserstein one.
For $\xi\geq \pi$ this is trivial since $\mathcal{W}_2^2=\xi^2\geq \pi^2>8=d^2$, where $d$ is accordingly computed using the \emph{stationary} cost \eqref{e:a2} because $\zeta\geq \pi$.
On the other hand for small $\xi< \pi$ the optimal strategy is the \emph{transport} one, and from \eqref{e:a0}-\eqref{e:a1} we get $\mathcal{W}_2^2-d^2=\xi^2-(8-8\cos(\xi/2))=o(\xi^3)$.
This means that our distance and the Wasserstein one agree at least at order two for short ranges, i-e
$$
\mathcal{W}_2(\delta_{x_0},\delta_{x_1})-d(\delta_{x_0},\delta_{x_1})=o(|x_1-x_0|^2)
$$
(both $\mathcal{W}_2$ and $d$ being of order $\xi=|x_1-x_0|$). 

\begin{rmk}
This threshold effect (at $\xi=\pi$) is natural since the cost of the stationary strategy \eqref{e:a2} is independent of the distance between the supports, while the pure transport strategy\eqref{kt}--\eqref{e:ccc} really does depend on how far mass is transported. 
This might be relevant in the context of image processing: our distance non-linearly interpolates between ``Wasserstein-like/transport'' at close range, and basically total variation for large distances.
A similar effect was discovered in \cite{PR14generalized} for the generalized Wasserstein distance considered therein. The bounded-Lipschitz distance also has such property.
\end{rmk}

\section{Application to a model of population dynamics}
\label{section:animals}

As an application of the abstract ideas from the previous sections we consider a problem in spatial population dynamics.  As originally proposed in \cite{mc90,cosner05}, the model describes a population of organisms inhabiting a domain $\Omega\subset \R^d$, whose macroscopic distribution $\rho(t,x)$ evolves according to reproduction and migration.  At each point $x\in\Omega$ lies a prescribed quantity of resources $m(x)$, and the local per capita reproduction rate is assumed to be proportional to the \emph{fitness} $m(x)-\rho$. In the absence of displacement of individuals the local population dynamics is thus given by the heterogeneous logistic equation $\frac{d\rho}{dt}=(m(x)-\rho)\rho$. The spatial migration is then dictated by the desire of individuals to maximize their local fitness in order to reproduce as best as possible, corresponding to an advection with local velocity $v=\nabla (m-\rho)$. Assuming that the animals cannot exit through walls and thus imposing a no-flux boundary condition, one 
obtains the following \emph{fitness-driven} model:
\begin{equation}
\label{eq:z}
\left\{
\begin{array}{ll}
\displaystyle
\p_t\rho = \operatorname{div} (\rho \nabla(\rho - m)) + \rho (m - \rho), & x \in \Omega, t 
> 0,
\\
\rho
\frac{
\partial \rho
}{
\partial \nu
}
=
\rho
\frac{
\partial m
}{
\partial \nu
}
,
&
x \in \partial \Omega, \ t > 0,
\\
\rho(0, x) = \rho_0(x), & x \in \Omega
.
\end{array}
\right.
\end{equation}
Here $\Omega \subset \mathbb R^d$ is a bounded, connected, open domain with smooth boundary,
$\nu$ is the outer normal to $\p \Omega$, $m(x)\geq 0$ is a given function (assumed to be at least $\mathcal{C}^{2,\alpha}(\o\Omega)$), and $\rho(t,x)$ is the unknown density of the population.  We consider nonnegative solutions of \eqref{eq:z}, so the initial datum $\rho_0(x)\in \mathcal{C}(\o \Omega)$ is 
supposed to be nonnegative as well.\\

In the previous sections we constructed our distance $d$ between measures on the whole space $\R^d$ by means of non-conservative continuity equations $\p _t\rho_t+\dive(\rho_t\nabla u_t)=\rho_t u_t$, which allowed to consider mass changes driven by the reaction term $\rho_t u_t$. Working in bounded domains $\Omega$ and imposing a no-flux condition on the velocity field $\nabla u_t$, so that the mass can \emph{only} change through the reaction term, the exact same construction formally gives a distance $d$ between measures in $\Omega$ (here we ignore all delicate issues related to boundary conditions and remain formal).
As a result all the Riemannian formalism in Section~\ref{section:Riemannian_structure} remains valid in the case of bounded domains, in particular formula \eqref{eq:formula_gradient_d} allows to compute gradients of functionals $\mathcal{F}:\mathcal{M}^+_b(\Omega)\to \R$. Anticipating that \eqref{eq:z} enjoys a comparison principle and that solutions  to the Cauchy problem remain nonnegative, 
we rather 
consider $\rho(t,x)$ as curves $t\mapsto \rho_t\in \mathcal{M}^+_b(\Omega)$. Using \eqref{eq:formula_gradient_d} it is easy to check, again formally, that \eqref{eq:z} can be written as the gradient flow
\begin{equation}
\label{eq:grad_flow_animals} 
 \frac{d\rho}{dt}=-\grad_d\mathcal{E}(\rho),\qquad \mathcal{E}(\rho)=\frac 12\int_{\Omega}|\rho(x)-m(x)|^2\rd x
\end{equation}
with respect to the distance $d$ (just write $\frac{\delta\mathcal{E}}{\delta\rho}=\rho-m$). In the sequel we will often refer to $\mathcal{E}$ as the \emph{entropy} functional, as is common for gradient flows in metric spaces.
In the view of \eqref{eq:grad_flow_animals} it is clear that the system is energetically driven by the fitness $m-\rho$, and it is natural to expect long-time convergence to the (unique) least-entropy stationary solution $\rho(t)\underset{t\to\infty}{\to} m$ (in some sense to be precised shortly). In fact this was proved in \cite{CW13} by means of comparison principles, but without convergence rate. Our main result in this section states that convergence is exponential:
\begin{theo}
\label{theo:animal_exponential_CV}
Assume that $m(x)>0$ in $\o\Omega$. Then for any nonnegative initial datum $\rho_0\in \mathcal{C}(\o\Omega)$, $\rho_0\not\equiv 0$, there is $\gamma\equiv \gamma(\rho_0,\Omega,m)>0$ such that the weak solution of \eqref{eq:z} satisfies
$$
\forall t\geq 0:\qquad \|\rho(t)-m\|_{L^2(\Omega)}\leq e^{-\gamma t}\|\rho_0-m\|_{L^2(\Omega)}.
$$
\end{theo}
\noindent The precise dependence of the rate $\gamma$ on the initial datum $\rho_0$ will be derived along the proof, see Remark~\ref{rmk:rate_depends_initial_datum} below.

Denoting the \emph{entropy dissipation} along trajectories by
$$
\mathcal{D}(\rho):=-\frac{d}{dt}\mathcal{E}(\rho)=\|\grad_d\mathcal{E}(\rho)\|^2,
$$
a classical way to retrieve exponential convergence $\rho(t)\to \rho_*$ for gradient flows is to obtain an \emph{entropy-entropy production} inequality
\begin{equation}
\label{eq:entropy_entropy_production}
\mathcal{D}(\rho)\geq \lambda(\mathcal{E}(\rho)-\mathcal{E}(\rho_*))
\end{equation}
for some $\lambda>0$.
This implies trend to equilibrium in the entropy sense $\mathcal{E}(\rho(t))\searrow \mathcal{E}(\rho_*)$ with exponential rate $\lambda$, and then using a Csicz\'ar-Kullback inequality one usually deduces that $\rho(t)\to\rho_*$ in some sense (typically $\|\rho(t)-\rho_*\|_{L^1(\Omega)}\to 0$).
We refer to \cite[Chapter 9]{villani03topics} for introductory material on this topic.

Note that the least entropy stationary solution is here $\rho_*=m(x)$ and that $\mathcal{E}(\rho_*)=0$. Formally multiplying \eqref{eq:z} by $\rho-m$ and integrating, the dissipation can be formally computed here as
$$
\mathcal{D}(\rho)=\|\rho-m\|^2_{H^1(\rd\rho)}=\int_\Omega (|\nabla(\rho-m)|^2+|\rho-m|^2)\rho\,\rd x,
$$
which is of course consistent with $\mathcal D=\|\grad_d \mathcal E\|^2$ and the Riemannian formalism $|\zeta|_{T_\rho\Mm}=\|\u\|_{H^1(\rd\rho)}$ from Section~\ref{section:Riemannian_structure}.
Observe that, due to the very definition of $\mathcal{E}(\rho)=\frac 12 \|\rho-m\|^2_{L^2(\Omega)}$, convergence in the entropy sense $\mathcal{E}(\rho(t))\to \mathcal E(m)=0$ already implies convergence $\rho(t)\to m$ in the stronger $L^2(\Omega)$ sense. As a consequence we can dispense from Csicz\'ar-Kullback inequalities and we only need an entropy-entropy production as in \eqref{eq:entropy_entropy_production}.
The latter classically holds (uniformly in the initial data) as soon as the entropy is $\lambda$-convex for some $\lambda>0$ (in the sense of Definition~\ref{definition:lambda_convexity}, see Section~\ref{section:Riemannian_structure} for second-order calculus). This can be tested here at least formally using the geodesic equations \eqref{e:geodeq}.
Unfortunately, for generic $m$ the specific entropy $\mathcal{E}(\rho)=\int_\Omega|\rho-m|^2$ does not appear to be $\lambda$-convex for any $\lambda>0$.
We use instead a generalized Beckner inequality (Theorem~\ref{th:eepi}), and as a consequence we only obtain exponential convergence for rates depending on the initial data as in Theorem~\ref{theo:animal_exponential_CV}.\\

From now on we claim mathematical rigor.
Following \cite{CW13} weak solutions are defined as
\begin{defi}
A nonnegative function
\begin{equation}
\label{eq:z-wsclass}
\rho \in
C([0,+\infty); L^1(\Omega))
\cap
L^2_\text{\rm loc} ([0, + \infty); H^1(\Omega))
\cap
L^\infty_\text{\rm loc}([0,+\infty)\times \overline \Omega)
\end{equation}
is called a \emph{global weak solution} of the Cauchy problem~\eqref{eq:z} if it satisfies the identity
\begin{multline}
\label{eq:z-ws}
-
\int_0^T
\int_\Omega
\rho \p_t\varphi
\, \rd x \, \rd t
+
\int_\Omega
\rho(T,x) \varphi(T,x)
\, \rd x
-\int_\Omega \rho_0(x) \varphi(0,x)
\, \rd x
\\
=
-
\int_0^T
\int_\Omega
\rho \nabla (\rho - m) \cdot \nabla \varphi
\, \rd x \, \rd t
+
\int_0^T
\int_\Omega
\rho (m - \rho) \varphi
\, \rd x \, \rd t
\end{multline}
for any $T>0$ and $\varphi \in C^1( [0, T] \times \overline\Omega)$.
\end{defi}

The existence and uniqueness for weak solutions of problem~\eqref{eq:z} were established in \cite{CW13}, as well as a comparison principle. We start by rigorously deriving the dissipation (entropy production) \emph{equality}:
\begin{lem}
\label{lem:H1*_dissipation}
Any weak solution actually satisfies $\rho\in \mathcal{C}([0,\infty);L^2(\Omega))$. The dissipation
$$
\mathcal{D}(t):=\int_\Omega \left(|\nabla(\rho(t)-m)|^2+|\rho(t)-m|^2\right)\rho(t)\,\rd x
$$
belongs to $L^1_{\textrm{loc}}([0,\infty))$, and
$$
\frac{d}{dt}\mathcal{E}(t)=-\mathcal{D}(t)
$$
in the sense of distributions $\mathcal{D}'(0,\infty)$.
\end{lem}
\begin{proof}
Since $\rho$ is uniformly bounded and $\rho,m\in L^2_{\textrm{loc}}([0,T);H^1(\Omega))$, it is straightforward to check that $\mathcal{D}\in L^1_{\textrm{loc}}([0,\infty))$. By density of $\mathcal{C}^1(\o\Omega)$ in $H^1(\Omega)$ and testing $\phi(t,x)=\eta(t)\psi(x)$ with $\eta\in \mathcal{C}^\infty_c(0,T)$ in the definition of weak solutions it is easy to see that $\p_t \rho\in L^2_{\textrm{loc}}([0,\infty);(H^1(\Omega))^*)$. By embedding $H^1(\Omega)\subset\subset L^2(\Omega)\subset (H^1(\Omega))^*$ one classically obtains $\rho\in\mathcal{C}([0,T];L^2(\Omega))$ for all $T>0$.

Since we know now that $\partial_t(\rho-m)=\p_t\rho\in L^2_{\textrm{loc}}([0,\infty);(H^1(\Omega))^*)$ we can legitimately take $\rho-m\in L^2_{\textrm{loc}}([0,\infty);H^1(\Omega))$ as a test function in \eqref{eq:z}, from which it classically follows that
\begin{align*}
\frac{d}{dt}\mathcal{E}(\rho(t))
& =\frac{d}{dt}\left(\frac 12 \|\rho(t)-m\|^2_{L^2(\Omega)}\right) =\left<\p_t\rho,\rho-m\right>_{(H^1)^*,H^1}\\
 & = -\int_{\Omega}\left(|\nabla(\rho(t)-m)|^2+|\rho(t)-m|^2\right)\rho(t)  = -\mathcal{D}(t)
\end{align*}
for a.e. $t\in (0,\infty)$ and in $L^1_{\textrm{loc}}([0,\infty))$ as desired.
\end{proof}
\begin{proof}[Proof of Theorem~\ref{theo:animal_exponential_CV}]
In order to apply our generalized Beckner inequality and retrieve an entropy-entropy production inequality, we first need to bound the $L^1$ mass of $\rho(t)$ from below.
To this end we let $\underline{\rho}_0(x)=\min\{m(x),\rho_0(x)\}$, and define $\underline{\rho}(t,x)$ to be the unique solution of the Cauchy problem with initial datum $\underline{\rho}_0$. Applying the comparison principle \cite[Lemma 4.2]{CW13} we have that
$$ 
\underline{\rho}(t,x)\leq \rho(t,x)\qquad \mbox{for a.e. } t,x
$$
and it suffices to show that $\int_\Omega \underline\rho(t)\geq c_0>0$.
Because $0$ and $m(x)$ are stationary solutions of \eqref{eq:z} (thus respectively sub and supersolutions) the comparison principle ensures that $0\leq \underline\rho(t,x)\leq m(x)$. Testing $\phi\equiv 1$ in the weak formulation we get $\frac{d}{dt}\int_\Omega \underline\rho(t)=\int_\Omega \underline\rho(t)(m-\underline\rho(t))\geq 0$, whence
$$
\int_\Omega \rho(t)\geq \int_\Omega \underline\rho(t)\geq \int_\Omega \underline\rho_0=:c_0>0
$$
as desired.

Since we are considering $m(x)>0$ we can apply the generalized Beckner inequality, Theorem~\ref{th:eepi} in the appendix, to get
\begin{align*}
\Phi\left(c_0\right)\int_\Omega |\rho(t)-m|^2
& \leq \Phi\left(\|\rho(t)\|_{L^1(\Omega)}\right)\cdot\int_\Omega |\rho(t)-m|^2\\
& \leq \int_{\Omega}\left(|\nabla(\rho(t)-m)|^2+|\rho(t)-m|^2\right)\rho(t).
\end{align*}
Wence by Lemma~\ref{lem:H1*_dissipation}
$$
\frac{d}{dt}\mathcal{E}(t)=-\mathcal{D}(t)\leq -2\gamma \mathcal{E}(t)\qquad \mbox{a.e. }t>0
$$
with $\gamma=\Phi(c_0)$. By standard Gr\"onwall arguments we conclude that
$$
\mathcal{E}(t)\leq e^{-2\gamma t}\mathcal{E}(0),
$$
which in turn yields $\|\rho(t)-m\|_{L^2(\Omega)}\leq e^{-\gamma t} \|\rho_0-m\|_{L^2(\Omega)}$
and achieves the proof.
\end{proof}
\begin{rmk}
\label{rmk:rate_depends_initial_datum}
From the proof above it is clear that the rate $\gamma>0$ in Theorem~\ref{theo:animal_exponential_CV} only depends on the initial datum through $c_0=\int_{\Omega}\min\{\rho_0,m\}\,\rd x$. 
\end{rmk}

\section{Appendix}
\label{section:appendix}
\subsection*{Lower-semicontinuity of the bounded-Lipschitz distance}
\begin{lem}\label{lem:d_BL_lsc_w*}
The bounded-Lipschitz distance $d_{BL}$ is sequentially lower semicontinuous with respect to the weak-$*$ topology. 
\end{lem}

The proof is obvious since the supremum in the definition of $d_{BL}$ can be restricted to smooth compactly supported functions, which are dense in $\mathcal C_0$.
\subsection*{A variant of the Banach-Alaoglu theorem}
\begin{prop}\label{Ban}
Let $(X,\|\cdot\|)$ be a separable normed vector space. Assume that there exists a sequence of seminorms $\{\|\cdot\|_k\}$ ($k=0,1,2,\dots$) on $X$ such that for every $x\in X$ one has 
$$
\|x\|_k\leq C \|x\|
$$
with a constant $C$ independent of $k,x$, and
$$
\|x\|_k\underset{k\to\infty}{\rightarrow} \|x\|_0.
$$
Let $\varphi_k$ ($k=1,2,\dots$) be a uniformly bounded sequence of linear continuous functionals on $(X,\|\cdot\|_k)$, resp., in the sense that\footnote{We recall that the continuous dual of a seminormed space is a Banach space, see \cite{bit}} 
$$
c_k:=\|\varphi_k\|_{(X,\|\cdot\|_k)^*}\leq C.
$$
Then the sequence $\{\varphi_k\}$ admits a converging subsequence $\varphi_{k_n}\to \varphi_0$ in the weak-$*$ topology of $X^*$, and
\begin{equation} \label{e:liminfp}
\|\varphi_0\|_{(X,\|\cdot\|_0)^*}\leq c_0:=\liminf\limits_k c_{k}.
\end{equation}
\end{prop}
\begin{proof} 
Since $$|\varphi_k(x)|\leq c_k \|x\|_k\leq C \|x\|$$ for every $x\in X$, $\{\varphi_k\}$ is a bounded sequence in $X^*$. Hence, by the Banach-Alaoglu theorem, it weakly-$*$ converges to some $\varphi_0\in X^*$ (up to a subsequence). Without loss of generality, $c_k\to c_0$, and passing to the limit in the inequality $|\varphi_k(x)|\leq c_k \|x\|_k$ we deduce \eqref{e:liminfp}.
\end{proof}
\subsection*{Arc-length reparametrization}
\begin{lem}
\label{lem:Lipschitz_time_arclength}
Let $\rho_0,\rho_1\in \Mm$ and $(\rho_t,\u_t)_{t\in [0,1]}$ be an admissible path connecting $\rho_0,\rho_1$ with finite energy $E[\rho;\u]$.
Then there exists an admissible time reparametrization $(\o\rho_t,\o \u_t)_{t\in[0,1]}$ with energy $E[\o\rho;\o \u]\leq E[\rho;\u]$ such that $\|\o \u\|^2_{H^1(\rd\o\rho_t)}=cst=E[\o \rho;\o \u]$ for a.e. $t\in [0,1]$.
\end{lem}
\begin{proof}
The argument is adapted from \cite[Lemma 1.1.4]{AGS06}.
Observing that $t\mapsto \|\u_t\|_{H^1(\rd\rho_t)}\in L^2(0,1)\subset L^1(0,1)$ we see that
$$
\mathtt{s}(t)=\int_0^t\|\u_\tau\|_{H^1(\rd\rho_\tau)}\,\rd\tau
$$
is absolutely continuous, nondecreasing, and
$$
\mathtt{s}(0)=0,\qquad \mathtt{s}(1)=\int_0^1\|\u_\tau\|_{H^1(\rd\rho_\tau)}\,\rd\tau=:L.
$$
The left-continuous inverse
$$
[0,L]\ni s\mapsto \mathtt{t}(s):=\min\{t\in [0,1]:\,\mathtt{s}(t)=s\}
$$
is a monotone increasing function, and has therefore countably many jumps.
Denoting $\lambda_t:=\|\u_t\|_{H^1(\rd\rho_t)}$ and observing that $\frac{d}{dt}\mathtt{s}(t)=\lambda_t$ we see that the countable set of discontinuities of $\mathtt{t}(s)$ is precisely the image by $\mathtt{s}$ of the critical points $\{t\in [0,1]:\,\lambda_t=0\}$, which by countability has zero $\mathrm{d}s$ measure in $[0,L]$. As a consequence $\lambda_{\mathtt{t}(s)}$ is positive $\mathrm{d}s$ a.e. in $[0,L]$, and
$$
s\in [0,L]:\qquad \tilde{\rho}_s:=\rho_{\mathtt{t}(s)}
\quad\mbox{and}\quad
\tilde{\u}_s:=\frac{\u_{\mathtt{t}(s)}}{\|\u_{\mathtt{t}(s)}\|_{H^1(\rd\rho_{\mathtt{t}(s)})}}
$$
are well-defined and measurable in $s$ with of course
$$
\|\tilde \u_s\|_{H^1(\rd\tilde\rho_s)}=1\quad \mbox{a.e. }s\in [0,L].
$$
Exploiting the narrow continuity of $t\mapsto \rho_t$ and $\mathtt{s}(\mathtt{t}(s))=s$ it is easy to see that $s\mapsto \tilde{\rho}_s$ is narrowly continuous and connects $\rho_0,\rho_t$ in time $s\in [0,L]$. Furthermore one can check that $\partial_s\tilde{\rho}_s+\dive(\tilde\rho_s\nabla\tilde u_s)=\tilde\rho_s \tilde u_s$ in the sense of distributions $\mathcal{D}'((0,L)\times\R^d)$, which formally follows by the chain rule
\begin{multline*}
 \partial_s\tilde\rho_s  =\partial_s(\rho_{\mathtt{t}(s)})=\frac{d}{ds}\mathtt{t}(s)\partial_t \rho_{\mathtt{t}(s)}=\frac{1}{\frac{d}{dt}\mathtt{s}(t)|_{\mathtt{t}(s)}}\partial_t \rho_{\mathtt{t}(s)}\\
  =\frac{1}{\|\u_{\mathtt{t}(s)}\|_{H^1(\rd\rho_{\mathtt{t}(s)})}}\left(-\dive\left(\rho_{\mathtt{t}(s)}\nabla u_{\mathtt{t}(s)}\right)+\rho_{\mathtt{t}(s)} u_{\mathtt{t}(s)}\right) \\
 = -\dive\left(\tilde\rho_s\nabla \tilde u_s\right)+\tilde\rho_s\tilde u_s
\end{multline*}
(this can be made rigorous using change of variables and $\|u_{\mathtt{t}(s)}\|_{H^1(d\rho_{\mathtt{t}(s)})}>0$ a.e. $s\in [0,L]$, see e.g. \cite[Lemma 8.1.3]{AGS06}).\\
As a consequence $(\tilde{\rho}_s,\tilde{\u}_s)_{s\in [0,L]}$ is an admissible curve connecting $\rho_0,\rho_1$ with energy
$$
E[\tilde\rho;\tilde \u]=\int_0^L\|\tilde \u_s\|^2_{H^1(\rd\tilde\rho_s)}\mathrm{d}s=\int_0^L1\,\mathrm{d}s=L.
$$
Setting $(\o \rho_t,\o \u_t)_{t\in [0,1]}:=(\tilde\rho_{tL},L\tilde \u_{tL})_{t\in [0,1]}$ in order to connect in time $t\in [0,1]$, the energy finally scales according to Remark~\ref{rmk:time_scaling} as
\begin{align*}
E[\o\rho;\o \u]& =L.E[\tilde{\rho};\tilde \u]=L^2\\
&=\left(\int_0^1\|\u_t\|_{H^1(\rd\rho_t)}\mathrm{d}t\right)^2\leq \int_0^1\|\u_t\|^2_{H^1(\rd\rho_t)}\mathrm{d}t=E[\rho;\u]
\end{align*}
as desired.
By construction we have that $\|\\o u_t\|_{H^1(\rd\o\rho_t)}$ is constant in time, and because $\int_0^1 \|\o \u_t\|_{H^1(\rd\o\rho_t)}^2\mathrm{d}t=E[\o\rho;\o \u]$ we conclude that $\|\o \u_t\|_{H^1(\rd\o\rho_t)}^2=E[\o\rho;\o \u]$ a.e. $t\in (0,1)$ and the proof is complete.
\end{proof}

\subsection*{Lower-semicontinuous translation of the Hopf-Rinow theorem} 
\begin{lem}
\label{hrt}
Let a metric space $(X,\varrho)$ be a complete length space. Assume that there exists a $\varrho$-boundedly compact Hausdorff topology $\sigma$ on $X$ (i-e $\varrho$-bounded sequences contain $\sigma$-converging subsequences)  such that $\varrho$ is sequentially lower semicontinuous with respect to $\sigma$. Then $(X,\varrho)$ is a geodesic space. 
\end{lem}

\begin{proof} Fix any two points $x,y\in X$. By \cite[Theorem 2.4.16(1)]{Bur}, it suffices to show that they admit a midpoint, i-e a point $z$ such that
\begin{equation*}
\label{mdp}
\varrho(x,y)= 2 \varrho(x,z)=2\varrho(z,y).
\end{equation*}
By \cite[Lemma 2.4.10]{Bur}, there exists a sequence $z_k$ of almost midpoints, i-e
\begin{equation*}
\label{amdp}
|\varrho(x,y)-2 \varrho(x,z_k)|\leq k^{-1},\quad |\varrho(x,y)-2 \varrho(y,z_k)|\leq k^{-1}.
\end{equation*} 
The sequence $\{z_k\}$ is $\varrho$-bounded, thus without loss of generality it $\sigma$-converges to some $z\in X$. Then $$2\varrho(x,z)\leq \lim_{k\to \infty} 2\varrho(x,z_k)= \varrho(x,y),$$
$$2\varrho(y,z)\leq \lim_{k\to \infty} 2\varrho(y,z_k)= \varrho(x,y).$$ But its is clear from the triangle inequality that the latter inequalities must be equalities. \end{proof}
\subsection*{A generalized Beckner inequality}
All the integrals below are implicitly computed with respect to the Lebesgue measure $\rd x$.
\begin{theo}
\label{th:eepi}
Let $\Omega$ in $\R^d$ be a bounded, connected, open domain, satisfying the 
cone property.
Let $m:\Omega\to \R$ be a Lipschitz function such that 
$\inf\limits_{x\in\Omega} m(x)>0$.
There exists a strictly increasing continuous 
scalar function $\Phi$ (depending merely on $\Omega$ and $m$) such that 
$\Phi(0)=0$  and
\begin{equation}
\label{e:eepi}
\Phi\left(\int_\Omega \rho\right)\int_\Omega |\rho-m|^2\leq 
\int_\Omega \rho|\rho-m|^2	+	\int_\Omega \rho|\nabla(\rho-m)|^2
\end{equation}
for every non-negative $\rho\in H^1\cap L^\infty(\Omega).$
\end{theo}
\begin{proof}
\emph{Step 1}. Without loss of generality, we may rescale the problem so that $\Omega$ has Lebesgue measure $1$.
Assume first that $m(x)\equiv 1$.
Under these assumptions, the generalized Beckner inequality \cite[Lemma 4]{chainais13entropy} with $p=3/2$, $q=4/3$ implies
\begin{equation} 
\label{e:eepi1}
\|\rho\|_{L^2(\Omega)}\left[ \int_\Omega 
\rho^2-\left(\int_\Omega \rho\right)^2\right] \leq C_\Omega\int_\Omega 
\rho|\nabla \rho|^2
\end{equation}
for every non-negative $\rho\in H^1\cap L^\infty(\Omega)$, where $C_\Omega$ depends only on $\Omega$.
Let $\lambda=\int_\Omega \rho \geq 0$.
If $\lambda=0$ then \eqref{e:eepi} trivially holds with $\Phi(0)=0$. 
If $\lambda=1$, then $\|\rho\|_{L^2(\Omega)}\geq 1$, 
and \eqref{e:eepi1} yields
\begin{equation*}
\label{e:eepi2}
\int_\Omega 
|\rho-1|^2	 \leq C_\Omega\int_\Omega \rho|\nabla \rho|^2,
\end{equation*}
which is even stronger than  \eqref{e:eepi}. 

\emph{Step 2}.
We now consider the case of arbitrary $\lambda>0$.
We set 
$\Phi(\lambda):=\min\left(\lambda, 
\frac{\lambda^2}{2C_\Omega}\right)$.
Let $\rho_\lambda=\rho/\lambda$.
Since we rescaled $|\Omega|=1$, by 
the H\"{o}lder inequality we have that
$$
\|\rho_\lambda\|_{L^3(\Omega)}\geq \|\rho_\lambda\|_{L^2(\Omega)}\geq 
\|\rho_\lambda\|_{L^1(\Omega)}=1,
$$
so
$$
\|\rho_\lambda\|_{L^3(\Omega)}^3\geq \|\rho_\lambda\|_{L^2(\Omega)}^2.
$$
Then we discover that 
\begin{multline*}
\int_\Omega |\rho-1|^2	=\int_\Omega |\lambda 
\rho_\lambda-1|^2	 =\lambda^2\left(\int_\Omega |\rho_\lambda|^2\right)-2\lambda +1\\
=
(\lambda^2-2\lambda)\left(\int_\Omega \rho_\lambda^2\right) +1 + 2\lambda  \int_\Omega 
|\rho_\lambda-1|^2\\
\leq \lambda^2\left(\int_\Omega \rho_\lambda^3\right) -2\lambda\left(\int_\Omega 
\rho_\lambda^2\right) +1 + 2\lambda  \int_\Omega |\rho_\lambda-1|^2\\
=\int_\Omega 
\rho_\lambda|\lambda \rho_\lambda-1|^2  + 2\lambda  \int_\Omega |\rho_\lambda-1|^2
\leq 
\int_\Omega  \rho_\lambda |\lambda \rho_\lambda-1|^2	  + 2\lambda  C_\Omega\int_\Omega 
\rho_\lambda|\nabla \rho_\lambda|^2\\
\leq\frac 1 {\Phi(\lambda)} \int_\Omega  \lambda 
\rho_\lambda |\lambda \rho_\lambda-1|^2	  + \frac {\lambda^3} {\Phi(\lambda)}\int_\Omega 
\rho_\lambda|\nabla \rho_\lambda|^2\\
= \frac 1 {\Phi\left(\int_\Omega 
\rho\right)}\left[ \int_\Omega \rho|\rho-1|^2	+	\int_\Omega \rho|\nabla 
\rho|^2\right].
\end{multline*}

\emph{Step 3}. We will now treat the case of generic $m(x)>0$ by a perturbation argument.
Throughout this step we denote by $C_m$ a generic constant depending only on $m$.
Let $\rho_m(x)=\rho(x)/m(x)$.
Then in the light of the previous step we see that 
\begin{multline*}
\Phi\left(\int_\Omega \rho\right)\int_\Omega |\rho-m|^2
=\Phi\left(\int_\Omega m\rho_m\right)\int_\Omega m^2|\rho_m-1|^2\\
\leq C_m\Phi\left(\int_\Omega \rho_m\right)	\int_\Omega |\rho_m-1|^2\\
\leq C_m\left[ \int_\Omega \rho_m|\rho_m-1|^2	+	\int_\Omega \rho_m|\nabla \rho_m|^2\right],
\end{multline*}
and it suffices to show that the latter sum does not exceed
$$
C_m\left[ \int_\Omega \rho|\rho-m|^2	+	\int_\Omega \rho|\nabla (\rho-m)|^2\right].
$$ 
Let $\alpha\in(0,1)$ be such that
$$
2\left(\frac 1 \alpha -1\right) \sup_{x\in\Omega}\left\{m(x)|\nabla m(x)|^2\right\} = \inf_{x\in\Omega}\left\{m^3(x)\right\}.
$$
Then we find that
\begin{multline*}
\int_\Omega \rho_m|\rho_m-1|^2	+\int_\Omega \rho_m|\nabla \rho_m|^2\\
\leq C_m\left[\frac 1 2\int_\Omega m^3\rho_m|\rho_m-1|^2	+	(1-\alpha)\int_\Omega m^3 \rho_m|\nabla \rho_m|^2\right]\\
\leq C_m\Big[\int_\Omega m^3 \rho_m |\rho_m-1|^2	 +(1-\alpha)\int_\Omega m^3 \rho_m |\nabla \rho_m|^2\\ 
+\left(1-\frac 1 \alpha\right)\int_\Omega m 	\rho_m|\rho_m-1|^2	|\nabla m|^2\Big]\\
\leq C_m\Bigg[\int_\Omega m^3 \rho_m |\rho_m-1|^2	 +\int_\Omega m^3 \rho_m |\nabla \rho_m|^2\\
+2\int_\Omega m^2 \rho_m |\rho_m-1|\nabla m\cdot\nabla \rho_m
+\int_\Omega m \rho_m |\rho_m-1|^2|\nabla m|^2\Bigg]\\
= C_m\left[\int_\Omega m^3 \rho_m |\rho_m-1|^2	+	\int_\Omega m \rho_m \left|(m \nabla \rho_m+\rho_m\nabla m)	-\nabla m\right|^2\right]\\
=C_m\left[ \int_\Omega \rho|\rho-m|^2	 +	\int_\Omega \rho|\nabla (\rho-m)|^2\right].
\end{multline*}
\end{proof}



%

\end{document}